\documentclass[a4paper,12pt]{article}
\usepackage[utf8]{inputenc}
\usepackage[T1]{fontenc}
\usepackage[normalem]{ulem} % pour barer le texte
\usepackage{amsmath, amssymb, amsthm, mathrsfs, todonotes, enumitem, stmaryrd, faktor, tikz-cd, enumitem, url, dsfont}
\usepackage[top=2.3cm,bottom=2cm,right=2cm,left=2cm]{geometry}
\usepackage{graphicx, float}
\usepackage{chngcntr}
\usepackage{color}
\definecolor{red}{rgb}{0.75,0,0}
\definecolor{vio}{rgb}{0.75,0,0.75}
\definecolor{green}{rgb}{0,0.75,0}

\tikzset{
    labl/.style={anchor=south, rotate=270, inner sep=.5mm}
}

\tikzset{
	lablbis/.style={anchor=south, rotate=20, inner sep=.5mm}
}

\numberwithin{equation}{section}
\counterwithin{figure}{section}

\newtheorem{theo}[equation]{Theorem}
\newtheorem{pgst}[equation]{Principal Genus Specialization Theorem}
\newtheorem{prop}[equation]{Proposition}
\newtheorem{lem}[equation]{Lemma}
\newtheorem{cor}[equation]{Corollary}
\newtheorem{quest}[equation]{Question}

\theoremstyle{definition}
\newtheorem{defi}[equation]{Definition}
\newtheorem{ex}[equation]{Example}

\theoremstyle{remark}
\newtheorem{rem}[equation]{Remark}

\newcommand{\density}[2]{$#1 _ #2$}

\newcommand*{\abs}[1]{\left\lvert #1\right\rvert} %valeur absolue
\newcommand*{\Z}{\ensuremath\mathbb{Z}} %entiers relatifs
\newcommand*{\Q}{\ensuremath\mathbb{Q}} %rationnels
\newcommand*{\R}{\ensuremath\mathbb{R}} %réels
\newcommand*{\A}{\ensuremath\mathbb{A}} 
\newcommand*{\K}{\ensuremath\mathbb{K}} %corps quelconque
\newcommand*{\PP}{\ensuremath\mathbb{P}} %proj
\newcommand*{\set}[1]{\left\{#1\right\}} %ensemble
\newcommand*{\st}{\enskip \middle |\enskip} %le ‘tel que ’ ; les espaces autour peuvent etre reduits , selon les gouts
\newcommand*{\rint}[1]{\mathcal{O}_{#1}} %anneau d'entiers
\newcommand*{\rints}[2]{\mathcal{O}_{#1,#2}} %anneau des S-entiers
\newcommand*{\slz}[1]{\SL_2(#1)} %sl_2
\newcommand*{\gltw}{\gl_2^{tw}} %gl_2^tw
\newcommand*{\cltw}{\cl^{tw}}
\newcommand*{\clztw}{\cl_{\Z}^{tw}}
\newcommand*{\clzsl}{\cl_{\Z}^{\SL_2}}
\newcommand*{\cltwdp}{\cl^{tw}_R(\Delta)}
\newcommand*{\clkx}{\cl^{tw}_{\K[X]}(4f)}
\newcommand*{\cloksx}{\cl^{tw}_{\oks[X]}(4f)}
\newcommand*{\cloks}{\cl^{tw}_{\oks}(4f(n))}
\newcommand*{\jk}{J(\K)}
\newcommand*{\ok}{\rint{\K}}
\newcommand*{\oks}{\rints{\K}{\mathcal{S}}}
\newcommand*{\okx}{\rint{\K}[X]}
\newcommand*{\oksx}{\oks[X]}
\newcommand*{\ws}{\mathcal{W}_{\mathcal{S}}}

\DeclareMathOperator{\SL}{SL}
\DeclareMathOperator{\gl}{GL}
\DeclareMathOperator{\spec}{Spec}

\DeclareMathOperator{\disc}{disc}

\DeclareMathOperator{\cl}{Cl}

\DeclareMathOperator{\pic}{Pic}
\DeclareMathOperator{\dv}{div}
\DeclareMathOperator{\lc}{lc}

\DeclareMathOperator{\val}{val}

\mathchardef\mhyphen="2D

\title{Quadratic forms and Genus Theory: a link with $2$-descent and an application to non-trivial specializations of ideal classes}
\author{William Dallaporta}
\date{April 2024}

\begin{document}

\maketitle

\begin{sloppypar}

\begin{small}\textit{Keywords:} binary quadratic form, Picard group, Genus Theory, $2$-descent on hyperelliptic curves, density on $\mathcal{S}$-integers \end{small}

\begin{small}\textit{MSC classes:} 11E16, 14H25, 14H40 (Primary); 11R45 (Secondary) \end{small}

\paragraph{Abstract} Genus Theory is a classical feature of integral binary quadratic forms. Using the author's generalization of the well-known correspondence between quadratic form classes and ideal classes of quadratic algebras, we extend it to the case when quadratic forms are twisted and have coefficients in any PID $R$. When ${R = \mathbb{K}[X]}$, we show that the Genus Theory map is the quadratic form version of the $2$-descent map on a certain hyperelliptic curve. As an application, we make a contribution to a question of Agboola and Pappas regarding a specialization problem of divisor classes on hyperelliptic curves. Under suitable assumptions, we prove that the set of non-trivial specializations has density $1$.

\section{Introduction}

It has been well-known since the work of Gauss in his \emph{Disquisitiones Arithmeticae} that, given ${\Delta \in \Z}$, the set
$$\set{\begin{array}{c} \text{equivalence~classes~of~primitive~binary~quadratic~forms} \\ ax^2+bxy+cy^2 \text{~with~} a,b,c \in \Z \text{~and~discriminant~} b^2-4ac = \Delta \end{array}}$$
can be endowed with a group structure, whose group operation is called the composition law (see Section~\ref{secdef} for the definitions).

Before going further, we must take care which notion of equivalence class we use. Over $\Z$, the natural action of $\SL_2(\Z)$ on quadratic forms is usually considered. In that setting, and when $\Delta$ is a negative integer, it is a classical fact that the above group (restricted to classes of positive definite quadratic forms) is isomorphic to the Picard group of the quadratic $\Z$-algebra of discriminant $\Delta$ (see \cite[Theorem~7.7]{Cox2} for a modern exposition). There have been numerous generalizations of this group structure and of this correspondence to other rings than $\Z$, possibly with a different action (see for example \cite{Towber} with $\SL_2$ or \cite{Kneser} with $\gl_2$). More recently, Wood gave a set-theoretical bijection over an arbitrary base scheme \cite{Wood}, and the present author derived from her work the sought group isomorphism, when $2$ is not a zero divisor on the base scheme \cite{moi}. In her work, Wood pointed out the importance of the \emph{twisted action} of $\gl_2$, which we denote by $\gltw$ (see Definition~\ref{action}). This is the only action we consider through this article, except in Subsection~\ref{subsecrz}, where we make the link with the classical $\SL_2$ action over $\Z$.

This general group isomorphism from \cite{moi} is stated over any base scheme $S$. Here, we shall consider affine schemes ${S = \spec(R)}$, where $R$ is an integral domain of characteristic different from $2$ such that every locally free $R$-module of finite rank is free. Under an additional assumption (see Proposition~\ref{propbij}), the set of (twisted-)equivalence classes of primitive binary quadratic forms with coefficients in $R$ and with discriminant ${\Delta \in R}$ is a group, which we denote by $\cltwdp$ (Proposition~\ref{propbij}). The neutral element of $\cltwdp$ is called the \emph{principal form class}.

Given a primitive binary quadratic form, it is natural to wonder if it lies in the principal form class or not. A classical feature of quadratic forms over $\Z$ is \emph{Genus Theory}, which partially answers this question and which is the main topic of this paper. Roughly speaking, the operation associating a class of quadratic forms to its set of values modulo its discriminant $\Delta$ yields a group homomorphism
$$\psi \colon \cl_{\Z}^{tw}(\Delta) \longrightarrow \faktor{\left(\faktor{\Z}{\Delta \Z}\right)^{\times}}{H_0}$$
whose kernel is called the \emph{principal genus} (Theorem~\ref{groupmorphism}). Here, $H_0$ denotes the set of values of the principal form class. In particular, a quadratic form whose class is not in the principal genus cannot be equivalent to the principal form.

In this article, we extend Genus Theory (for the twisted action $\gltw$) to quadratic forms over principal ideal domains. In this general setting, it is already a difficult problem to determine precisely the principal genus. A simple argument shows that it always contains the subgroup of squares. Over $\Z$, when the discriminant is negative, we show in Proposition~\ref{propcasez} that the converse is true (this is just an adaptation in our context of the proofs of the classical results).

We then study the case when the base ring is $\K[X]$ (where $\K$ is a field of characteristic $0$), and when the discriminant is of the form ${\Delta = 4f}$ with ${f \in \K[X]}$ a square-free monic polynomial of odd degree at least $3$. In this situation, the group of (twisted)-equivalence classes of quadratic forms of discriminant $4f$ is isomorphic to the group of $\K$-points of the Jacobian variety of the hyperelliptic curve $\mathcal{C}$ defined over $\K$ by the equation ${Y^2 = f(X)}$. This correspondence is closely related to Mumford's description of the Jacobian, and was already used by Gillibert in that setting \cite{Jean}. We prove in Subsection~\ref{subsec2descente} that Genus Theory over $\K[X]$ turns out to be the quadratic form version of the $2$-descent map on the Jacobian of $\mathcal{C}$. More precisely, by combining Proposition~\ref{explx-t} and Theorem~\ref{inj}, we obtain

\begin{theo} \label{theogoal2desc}
	Let $\K$ be a field of characteristic $0$, let ${f \in \K[X]}$ be a square-free monic polynomial of odd degree at least $3$, let ${L := \faktor{\K[X]}{\left\langle f(X) \right\rangle}}$, and let $J$ be the Jacobian variety of the hyperelliptic curve defined by the affine equation ${Y^2=f(X)}$ over $\K$. Let us denote by $\Psi$ the Genus Theory homomorphism \eqref{mappsi} and by $\lambda$ the $2$-descent map on $J(\K)$. Then the following diagram commutes
	\begin{equation} \label{diag2descintro}
	\begin{tikzcd}
	\faktor{\clkx}{\clkx^{\square}} \arrow[r,"\sim "] \arrow[d,"\Psi "'] & \faktor{\jk}{2\jk} \arrow[d,"\lambda"] \\
	\faktor{L^{\times}}{\K^{\times}L^{\times \square}} & \faktor{L^{\times}}{L^{\times \square}} \arrow[l,"pr "]
	\end{tikzcd}
	\end{equation}
	where the exponent $\square$ denotes the subgroup of squares, and $pr$ is the natural projection. Furthermore, $\Psi$ is injective; in other words, the principal genus is precisely the subgroup of squares.
\end{theo}

The fact that our base ring is a principal ideal domain is heavily used to find an adequate representative of a given class of quadratic forms (Lemma~\ref{coprime}). This is a property which is at the heart of most of the technical arguments. If one wants to extend Genus Theory to quadratic forms over more general rings than PIDs, then one must in particular extend Lemma~\ref{coprime} or find a way to deal without it.

An as application of Genus Theory, the last Section of our article is devoted to the following question, which is closely related to a question raised by Agboola and Pappas \cite{AgboolaPappas}.

\begin{quest} \label{quest}
	Let $\K$ be a number field. Let $\mathcal{C}$ be a hyperelliptic curve over $\K$ of genus ${g \geq 1}$, with a $\K$-rational Weierstrass point. Let us choose an affine equation of $\mathcal{C}$ of the form ${Y^2=f(X)}$ where ${f \in \ok[X]}$ is a square-free monic polynomial of odd degree ${2g+1}$. Let ${I \in \pic\left(\faktor{\ok[X,Y]}{\left\langle Y^2-f(X) \right\rangle}\right)}$ be a non-trivial ideal class. Can we find ${n \in \ok}$ such that the specialization of $I$ at ${X = n}$ gives a non-trivial ideal class $I_n$ in $\pic\left(\rint{\K(\sqrt{f(n)})}\right)$, or at least in $\pic\left(\faktor{\ok[Y]}{\left\langle Y^2-f(n) \right\rangle}\right)$\,?
\end{quest}

We answer positively the second part of Question~\ref{quest} for ideal classes $I$ which are not squares, at least after inverting a finite number of prime ideals of $\ok$. We further prove that the density of non-trivial specializations is $1$, for any ``reasonable'' density. In the case of square ideal classes, our arguments which rely on Genus Theory cannot be extended, since squares are already in the principal genus.

Regarding hyperelliptic curves, several results have already been established about non-trivial specializations:
\begin{itemize}
	\item when $\mathcal{C}$ is an elliptic curve over ${\K = \Q}$ and $I$ has infinite order, Soleng proved that there exist infinitely many non-trivial specializations $I_n$ in imaginary quadratic extensions of $\Q$ whose order is unbounded as $n$ goes to infinity \cite[Theorem~4.1]{Soleng};
	
	\item when ${\K = \Q}$ and $I$ has finite order, Gillibert and Levin used Kummer Theory and Hilbert's Irreducibility Theorem to show that, after inverting primes of bad reduction, there exist infinitely many non-trivial specializations $I_n$ in imaginary quadratic extensions of $\Q$ \cite[Corollary~3.8]{GillibertLevin};
	
	\item when ${\K = \Q}$ and $I$ has infinite order, Gillibert showed that there exist infinitely many negative integers $n$ such that $I_n$ is a non-trivial ideal class of the order $\Z[\sqrt{f(n)]}$. With the additional assumption that the irreducible factors of $f$ all have degree at most $3$, this leads to infinitely many non-trivial ideal classes in $\pic\left(\rint{\Q(\sqrt{f(n)})}\right)$ \cite[Theorems 1.2 and 1.3]{Jean}. Among Gillibert's main ingredients, one can find Wood's correspondence with binary quadratic forms and a generalization of Soleng's argument.
\end{itemize}

Let us assume that ${\K = \Q}$ for the time being. Given a non-trivial ideal class $I$ of $\faktor{\Z[X,Y]}{\left\langle Y^2-f(X) \right\rangle}$, can we find non-trivial specializations of $I$ in real quadratic extensions of $\Q$\,? Soleng's argument relies on properties which are specific to negative discriminants, and do not generalize to the positive case. This leads us to consider a different approach.

Numerical experiments show that, depending on the ideal class $I$ we start from, there may exist congruence classes of ${n \in \Z}$ leading to non-trivial specializations (see Example~\ref{sieve}), whatever the sign of the discriminant.

As in the work of Gillibert, we use Wood's bijection between invertible ideal classes of quadratic algebras and equivalence classes of primitive binary quadratic forms as described in \cite[Corollary~3.25]{moi}. In this setting, the ideal class $I$ we start from corresponds to the equivalence class of a primitive quadratic form ${q(x,y) = ax^2+bxy+cy^2}$ with discriminant ${b^2-4ac = 4f}$, where ${a,b,c \in \ok[X]}$. Question~\ref{quest} now asks whether one can find ${n \in \ok}$ such that the specialized quadratic form ${a(n)x^2 + b(n)xy + c(n)y^2}$ is not equivalent to the principal form ${x^2-f(n)y^2}$, that is, the class of quadratic forms corresponding to the trivial ideal class in Wood's bijection.

The presentation of Genus Theory in this article requires us to work over a principal ideal domain. As $\ok$ may not be a PID, we slightly modify it by inverting finitely many prime ideals. Thus, we will work over $\oks$ instead of $\ok$, where $\oks$ is the \emph{ring of $\mathcal{S}$-integers of $\K$}. Despite the fact that $\oks[X]$ is not a PID, by making a suitable choice of $\mathcal{S}$, one can relate classes of quadratic forms over $\oks[X]$ with classes of quadratic forms over $\K[X]$ (Proposition~\ref{propoksisok}). Notice that in the terminology of divisors, this operation is the restriction to the generic fibre.

We then prove that, given a class $q$ of quadratic forms over $\oks[X]$ which is not in the principal genus when viewed over $\K[X]$, there exist infinitely many ${n \in \oks}$ such that the specialized class $q_n$ of quadratic forms is not in the principal genus. We achieve this in Theorem~\ref{theocritmod}. Together with Theorem~\ref{theodens1} and Remark~\ref{choice}, a complete version of the main result is the following.

\begin{theo} \label{theogoaldens}
	Let $\K$ be a number field. Let ${f \in \ok[X]}$ be a square-free monic polynomial of odd degree at least $3$. Let $\mathcal{S}$ be a finite set of nonzero prime ideals of $\ok$ such that $\oks$ is a PID. Let ${I \in \pic\left(\faktor{\oks[X,Y]}{\left\langle Y^2-f(X) \right\rangle}\right)}$ be a non-trivial ideal class.
	
	Assume that the ideal class generated by $I$ in $\pic\left(\faktor{\K[X,Y]}{\left\langle Y^2-f(X) \right\rangle}\right)$ is not a square. Then the set of ${n \in \oks}$ such that the specialization of $I$ at ${X=n}$ gives a non-trivial ideal class $I_n$ of $\faktor{\oks[Y]}{\left\langle Y^2-f(n) \right\rangle}$ has density $1$, for any density on $\oks$ as in Definition~\ref{defidensity}. In particular, there are infinitely many ${n \in \oks}$ such that $I_n$ is non-trivial.
\end{theo}

If we choose $\mathcal{S}$ which contains the prime ideals dividing $2\disc(f)$ and such that $\oks$ is a PID, then Theorem~\ref{theogoaldens} gives a partial answer to a question raised by Agboola and Pappas \cite{AgboolaPappas}. More precisely, following \cite[\S 2.1]{Jean}, there exists a smooth projective model ${\overline{\mathcal{W}} \longrightarrow \spec(\oks)}$ of $\mathcal{C}$ such that, set-theoretically,
$$\overline{\mathcal{W}} = \spec\left(\faktor{\oks[X,Y]}{\left\langle Y^2-f(X) \right\rangle}\right) \cup \overline{\set{\infty}}$$
together with an isomorphism ${\pic\left(\faktor{\oks[X,Y]}{\left\langle Y^2-f(X) \right\rangle}\right) \simeq \pic^0(\overline{\mathcal{W}})}$, where $\overline{\set{\infty}}$ is the scheme-theoretic closure of the point at infinity of $\mathcal{C}$. The result of Theorem~\ref{theogoaldens} implies that a degree $0$ line bundle on $\overline{\mathcal{W}}$ which is not a square has infinitely many non-trivial specializations over suitable quadratic $\oks$-orders.

\paragraph{Notations.}

All through this paper, the rings we consider are commutative and endowed with a multiplicative identity denoted by $1$. If $R$ is a ring, $R^{\times}$ denotes its group of units. Given ${r_1,\ldots,r_m \in R}$, the ideal generated by ${r_1,\ldots,r_m}$ is denoted by ${\left\langle r_1,\ldots,r_m \right\rangle}$. If $G$ is a group, then $G^{\square}$ denotes its subgroup of squares (thinking of the group law multiplicatively).

Given two integers ${a<b}$, we denote by ${\llbracket a,b \rrbracket}$ the set of integers $n$ such that ${a \leq n \leq b}$.

If $T$ is a scheme, we denote by $\pic(T)$ the Picard group of $T$. When $T$ is Noetherian and reduced, we shall identify $\pic(T)$ with the group of Cartier divisors modulo linear equivalence. When $R$ is a domain, we write by abuse of notations $\pic(R)$ instead of $\pic(\spec(R))$, which we also identify with the group of invertible fractional ideals modulo principal ones.

We refer the reader to Definition~\ref{action} and Proposition~\ref{propbij} for the meaning of $\gltw$ and of $\cltwdp$ respectively.

\paragraph{Acknowledgements.}

This work is part of my PhD at the Institut de Mathématiques de Toulouse. I am grateful to Sander Mack-Crane, Ignazio Longhi, Florent Jouve and Yuri Bilu for helpful discussions about densities over rings of $\mathcal{S}$-integers. I also address special thanks to Jean Gillibert and Marc Perret who made this work possible, who regularly gave precious advice, and who were particularly encouraging all along this work. The final writing of this paper owes a lot to the meticulous proofreading of Christian Wuthrich and of the anonymous referee, whom I warmly thank.

\section{Binary quadratic forms and Picard groups of quadratic algebras} \label{secdef}

The kind of rings $R$ we consider in this paper are mostly of the form $R'$ or $R'[X]$ with $R'$ a principal ideal domain of characteristic different from $2$. An important property they share is the fact that every locally free $R$-module of finite rank must be free, according to \cite{Seshadri}.

\begin{defi}
	Let $R$ be a ring such that every locally free $R$-module of finite rank is free. A (binary) \emph{quadratic form} $q$ over $R$ is a homogeneous degree $2$ polynomial in $R[x,y]$. It is of the form ${ax^2+bxy+cy^2}$ for some ${a,b,c \in R}$, and is denoted by $[a,b,c]$. It is called \emph{primitive} if the ideal generated by $a,b,c$ in $R$ is the unit ideal. Its \emph{discriminant} is the quantity ${\Delta := b^2-4ac \in R}$.
\end{defi}

\begin{defi} \label{action}
	Let ${M = \begin{pmatrix} \alpha & \beta \\ \gamma & \delta \end{pmatrix} \in \gl_2(R)}$. Following \cite{Wood}, we define the \emph{twisted action of $\gl_2(R)$} over the set of quadratic forms via ${M \cdot q := \frac{1}{\det(M)}(q \circ M)}$, that is
	$$\left(\begin{pmatrix} \alpha & \beta \\ \gamma & \delta \end{pmatrix} \cdot q \right) (x,y) := \frac{1}{\alpha \delta - \beta \gamma} q(\alpha x + \beta y, \gamma x + \delta y) \hspace{1cm} \forall ~ x,y \in R.$$
	We denote this action by $\gltw$. Two quadratic forms $q$ and $q'$ are \emph{$\gltw$-equivalent} (\emph{equivalent} for short) if there exists $M \in \gl_2(R)$ such that ${q' = M \cdot q}$. In the following, a \emph{class of quadratic forms} $[a,b,c]$ refers to the equivalence class of the quadratic form $[a,b,c]$.
\end{defi}

\begin{rem} \label{remcompute}
	With the same notations, if ${q = [a,b,c]}$, then
	\begin{align*}
	\begin{pmatrix} \alpha & \beta \\ \gamma & \delta \end{pmatrix} \cdot [a,b,c] & = \frac{1}{\alpha\delta-\beta\gamma} [a \alpha^2 + b\alpha \gamma + c \gamma^2, b(\alpha \delta + \beta \gamma) + 2(a \alpha \beta + c \gamma \delta), a \beta^2 + b \beta \delta + c \delta^2] \\
	& = \frac{1}{\alpha\delta-\beta\gamma}[q(\alpha,\gamma),q(\alpha+\beta,\gamma+\delta)-q(\alpha,\gamma)-q(\beta,\delta),q(\beta,\delta)].
	\end{align*}
\end{rem}

We recall the main link between quadratic forms and ideals in our setting (\cite[Corollary~3.25]{moi}).

\begin{prop} \label{propbij}
	Let $R$ be an integral domain of characteristic different from $2$ such that every locally free $R$-module of finite rank is free. Let ${\Delta \in R}$ be such that the equation ${\Delta \equiv x^2 \pmod{4R}}$ has a unique solution $x$ modulo $2R$, and let $\pi$ be any lift of $x$ to $R$. Denote by $\cltw_R(\Delta)$ the set of $\gltw$-equivalence classes of primitive quadratic forms over $R$ with discriminant $\Delta$. Then we have a bijection
	\begin{equation} \label{bij}
	\begin{aligned}
	\cltw_R(\Delta) & \overset{1 \colon 1}{\longleftrightarrow} \pic\left(\faktor{R[\omega]}{\left\langle \omega^2 + \pi \omega - \frac{\Delta-\pi^2}{4} \right\rangle}\right) \\
	\bigg[[a,b,c] \text{~with~} a\neq 0\bigg] & \longmapsto \bigg[\langle \omega + \frac{\pi-b}{2}, a \rangle\bigg] \end{aligned}.
	\end{equation}
	This allows us to endow $\cltw_R(\Delta)$ with a group structure, whose operation $\ast$ is called the \emph{composition law}.
\end{prop}

\begin{rem}
	In general, one must care about the existence of different solutions ${x \pmod{2R}}$ of the equation ${\Delta \equiv x^2 \pmod{4R}}$ in Proposition \ref{propbij}, for instance when ${R = \Z[\sqrt{8}]}$ (\cite[Example~2.28]{moi}). The value of such an $x$ modulo $2R$ is called the \emph{parity} and is an invariant of our quadratic forms. For the rings $R$ we shall consider, the parity is completely determined by the discriminant $\Delta$, that is, there is a unique solution ${x \pmod{2R}}$. Indeed, if $R$ is a PID, or satisfies either the fact that $2$ is a unit or the fact that $\left\langle 2 \right\rangle$ is a prime ideal, as will always be the case in the following, then uniqueness is guaranteed by \cite[Proposition~2.26]{moi}.
\end{rem}

\begin{defi} \label{defiprincipalform}
	The \emph{principal form class} is the neutral element of $\cltw_R(\Delta)$. A representative of this class is ${[1,\pi,-\frac{\Delta-\pi^2}{4}]}$, for every ${\pi \in R}$ such that ${\Delta \equiv \pi^2 \pmod{4R}}$.
\end{defi}

Composition of quadratic forms classes over $\Z$ has been well-known since Gauss, and has already been extended to other rings. For example, Cantor described it over ${R = \K[X]}$ for every field $\K$ of characteristic different from $2$ (\cite[\S 3]{Cantor}). At least over $\Z$, there exist various formulae to compute the composition of two given quadratic forms. We extend one such formula to the case when $R$ is a PID, without any particular difficulty. However, our formula requires a coprimality condition on the first coefficients of the quadratic forms, which will be enough for our purpose and easily fulfilled (Lemma~\ref{coprime}).

\begin{prop} \label{compo}
	Let $R$ be a PID with ${2 \neq 0}$. Let ${q_1 = [a_1,b_1,c_1], q_2 = [a_2,b_2,c_2] \in \cltwdp}$, and assume that $a_1$ and $a_2$ are nonzero and coprime. Let ${a_1r_1 + a_2r_2 = 1}$ be a Bézout relation. Then, the composition ${q_1 \ast q_2}$ of $q_1$ and $q_2$ is the class of ${[a_1a_2,B, -\frac{\Delta-B^2}{4a_1a_2}]}$, where ${B = a_1r_1b_2 + a_2r_2b_1}$.
\end{prop}

\begin{rem}
	When ${R = \Z}$, the quadratic form ${[a_1a_2,B, -\frac{\Delta-B^2}{4a_1a_2}]}$ with $B$ as above is known as the \emph{Dirichlet composition} of $q_1$ and $q_2$ (\cite[(3.7)]{Cox2}).
\end{rem}

\begin{proof}
	Denote ${\alpha_i := \frac{\pi-b_i}{2}}$, where ${\pi \in R}$ is any element such that ${\Delta \equiv \pi^2 \pmod{4R}}$. By definition of $\ast$ from Proposition~\ref{propbij}, using the same notations, ${q_1 \ast q_2}$ corresponds to the product
	$$\left\langle a_1, \omega + \alpha_1 \right\rangle \left\langle a_2, \omega + \alpha_2 \right\rangle = \left\langle a_1a_2,a_1(\omega+\alpha_2),a_2(\omega + \alpha_1), -\frac{b_1+b_2}{2}\omega + \alpha_1\alpha_2+\frac{\Delta-\pi^2}{4} \right\rangle.$$
	
	Observe the relation
	$$\begin{pmatrix}
	1 & 0 & 0 & 0 \\
	0 & a_2 & -a_1 & 0 \\
	0 & r_1\frac{b_1+b_2}{2} & r_2\frac{b_1+b_2}{2} & 1 \\
	0 & -r_1 & -r_2 & 0
	\end{pmatrix} \begin{pmatrix} a_1a_2 \\ a_1(\omega+\alpha_2) \\ a_2(\omega + \alpha_1) \\ -\frac{b_1+b_2}{2}\omega + \alpha_1\alpha_2+\frac{\Delta-\pi^2}{4} \end{pmatrix} = \begin{pmatrix} a_1a_2 \\ a_1a_2\frac{b_1-b_2}{2} \\ -a_1a_2(r_2c_1 + r_1c_2) \\ -\omega - \frac{\pi - B}{2} \end{pmatrix}.$$
	Since the square matrix on the left hand side has determinant $1$, the above ideal is the same as the one spanned by the elements of the vector on the right hand side. Therefore,
	$$\left\langle a_1, \omega + \alpha_1 \right\rangle \left\langle a_2, \omega + \alpha_2 \right\rangle = \left\langle a_1a_2, \omega + \frac{\pi-B}{2} \right\rangle.$$
	The corresponding quadratic form in bijection~\eqref{bij} is the class of ${[a_1a_2,B, -\frac{\Delta-B^2}{4a_1a_2}]}$, concluding the proof.
\end{proof}

\begin{rem}
	One can also prove Proposition~\ref{compo} with a composition formula in Gauss' style, checking that the relation
	\begin{equation} \label{compoformula}
	(a_1x_1^2 + b_1x_1y_1 + c_1y_1^2)(a_2x_2^2 + b_2x_2y_2 + c_2y_2^2) = a_1a_2X^2 + BXY + \frac{B^2-\Delta}{4a_1a_2}Y^2
	\end{equation}
	is true, where
	\begin{align*}
	X & = x_1x_2 + r_2\frac{b_2-b_1}{2}x_1y_2 + r_1\frac{b_1-b_2}{2}x_2y_1 - (r_2c_1 + r_1c_2)y_1y_2 \\
	\text{and} ~~~~ Y & = a_1x_1y_2 + a_2x_2y_1 + \frac{b_1+b_2}{2}y_1y_2.
	\end{align*}
\end{rem}

\begin{rem}
	If $a_1$ and $a_2$ are not supposed to be coprime, we are still able to give a composition algorithm over a PID, but the gcd of $a_1, a_2$ and $\frac{b_1+b_2}{2}$ shows up and $B$ must be modified (see \cite[Theorem~4.10]{Buell}).
\end{rem}

\section{Genus theory over a PID} \label{secgentheory}

Throughout this Section, $R$ is a PID of characteristic different from $2$. In this paper, Genus Theory will always refer to the construction of a particular group homomorphism from $\cltwdp$ to a quotient of $\left(\faktor{R}{\Delta R}\right)^{\times}$. This homomorphism essentially maps a quadratic form $q$ to its set of values modulo the discriminant $\Delta$. This construction was introduced by Gauss over $\Z$, and we shall check that it extends to arbitrary PIDs.

All the techniques of Subsection~\ref{subsecgencase} are classical, mainly adapted from the case of $\Z$ treated in \cite[\S 1-\S 3]{Cox2}. However, we must keep in mind a difference about the quadratic forms we consider: Genus Theory over $\Z$ is usually done with $\SL_2(\Z)$-equivalence classes, while we are dealing with $\gltw$-ones. We describe the connection between these two cases in Subsection~\ref{subsecrz}.

\subsection{The general case} \label{subsecgencase}

\begin{defi} \label{defisetofvalues}
	Let $q$ be a primitive quadratic form of discriminant ${\Delta \in R}$. Its \emph{set of values in $\left(\faktor{R}{\Delta R}\right)^{\times}$} is given by
	$$\val_{\Delta}(q) := \set{q(x,y) \st x,y \in \faktor{R}{\Delta R}} \cap \left(\faktor{R}{\Delta R}\right)^{\times}.$$
	
	The \emph{set of values in $\left(\faktor{R}{\Delta R}\right)^{\times}$ of the class of $q$} is ${H_q := \displaystyle{\bigcup_{q' \sim q} \val_{\Delta}(q')}}$.
\end{defi}

\begin{rem} \label{remHq}
	Two equivalent quadratic forms represent the same values up to units of $R$, because of the twist by the determinant of the acting matrix. Therefore, for all ${q \in \cltwdp}$, we have ${H_q = R^{\times}\val_{\Delta}(q)}$, where by abuse of notations $R^{\times}$ stands for its image by the canonical projection ${R \longrightarrow \faktor{R}{\Delta R}}$.
\end{rem}

\begin{prop} \label{H0}
	Let $H_0$ be the set of values taken by the principal form class in $\left(\faktor{R}{\Delta R}\right)^{\times}$. Then $H_0$ is a subgroup of $\left(\faktor{R}{\Delta R}\right)^{\times}$, containing the squares. Moreover, if $2$ is invertible in $\faktor{R}{\Delta R}$, then ${H_0 = R^{\times}\left(\faktor{R}{\Delta R}\right)^{\times \square}}$, where the exponent ${}^{\square}$ denotes the subgroup of squares.
\end{prop}

\begin{proof}
	As already noticed in Remark~\ref{remHq}, ${H_0 = R^{\times}\val_{\Delta}(q_0)}$ where ${q_0 = [1,\pi,-\frac{\Delta - \pi^2}{4}]}$ is a representative of the principal form class as in Definition~\ref{defiprincipalform}. We focus on $\val_{\Delta}(q_0)$.
	
	It is a straightforward computation to check the following equation, which is a particular case of the composition formula (\ref{compoformula}):
	$$\left(x_1^2+\pi x_1y_1 - \frac{\Delta-\pi^2}{4}y_1^2\right)\left(x_2^2+\pi x_2y_2 - \frac{\Delta-\pi^2}{4}y_2^2\right) = X^2+\pi XY - \frac{\Delta-\pi^2}{4}Y^2$$
	for all ${x_1,x_2,y_1,y_2 \in R}$, where ${X := x_1x_2 + \frac{\Delta-\pi^2}{4}y_1y_2}$ and ${Y := x_1y_2+x_2y_1+\pi y_1y_2}$. This formula proves that $\val_{\Delta}(q_0)$ is stable under multiplication. Furthermore, it is also stable under inversion: if ${\alpha \in \val_{\Delta}(q_0)}$, then there exist ${x,y \in \faktor{R}{\Delta R}}$ such that ${\alpha \equiv x^2+\pi xy - \frac{\Delta-\pi^2}{4}y^2 \pmod{\Delta}}$, hence
	$$\alpha^{-1} \equiv (x\alpha^{-1})^2+\pi (x\alpha^{-1})(y\alpha^{-1}) - \frac{\Delta-\pi^2}{4}(y\alpha^{-1})^2 \in \val_{\Delta}(q_0).$$
	Therefore, $\val_{\Delta}(q_0)$ and $H_0$ are indeed subgroups of $\left(\faktor{R}{\Delta R}\right)^{\times}$.
	
	Clearly, $\val_{\Delta}(q_0)$ contains the squares since ${[1,\pi,-\frac{\Delta-\pi^2}{4}](x,0) = x^2}$ for all $x$. If $2$ is invertible modulo $\Delta$, then for all ${x,y \in R}$, we have
	$$\left[1,\pi,-\frac{\Delta-\pi^2}{4}\right](x,y) \equiv \left(x+\frac{\pi}{2} y\right)^2 \pmod{\Delta},$$
	hence the result.
\end{proof}

The key point of most of the following results is the next Lemma, which makes heavy use of factorization and Bézout relations. Its main outcome for our purpose is the Bézout relation it induces.

\begin{lem} \label{coprime}
	Let ${q \in \cltwdp}$ and let ${h \in R \setminus \set{0}}$. Then $q$ is equivalent to some form $[a,b,c]$ where ${a \neq 0}$ is coprime to $h$.
\end{lem}

\begin{proof}	
	Denote ${q = [a_0,b_0,c_0]}$. First, note that if there exist ${x_0,y_0 \in R}$ coprime such that ${q(x_0,y_0) = a}$, then $q$ is equivalent to $[a,b,c]$ for some ${b,c \in R}$. Indeed, if we find such $x_0$ and $y_0$, then there exist ${\beta, \delta \in R}$ such that ${x_0 \delta - y_0 \beta = 1}$ (since $R$ is a PID). By Remark~\ref{remcompute}, we deduce that ${\begin{pmatrix} x_0 & \beta \\ y_0 & \delta \end{pmatrix} \cdot [a_0,b_0,c_0] = [a,b,c]}$ for some ${b,c \in R}$, where $a$ is coprime to $h$.
	
	As there is nothing to prove if ${h \in R^{\times}}$, we may assume that $h$ is not a unit. Decompose it as a product of irreducibles: ${h = \underset{i}{\prod} p_i^{r_i}}$ where $p_i$ is prime and ${r_i \geq 1}$ for all $i$. Since $q$ is primitive, a given $p_i$ cannot divide ${q(1,0) = a_0}$, ${q(0,1) = c_0}$ and ${q(1,1) = a_0+b_0+c_0}$ at the same time. Hence, for all $i$, there exist $x_i,y_i$ coprime elements of $R$ such that ${q(x_i,y_i) \not\equiv 0 \pmod{p_i}}$. We lift the pairs ${((x_i,y_i) \pmod{p_i})}$ with the Chinese Remainder Theorem to some ${(x,y) \in R^2}$, and we set ${(x_0,y_0) := \left(\frac{x}{\gcd(x,y)},\frac{y}{\gcd(x,y)}\right)}$. Then $x_0$ and $y_0$ are coprime, and ${a:=q(x_0,y_0)}$ is nonzero and coprime to $h$.
\end{proof}

\begin{rem} \label{remsl2}
	We actually proved that every primitive quadratic form over $R$ is $\SL_2(R)$-equivalent to one whose first coefficient is coprime to $h$.
\end{rem}

\begin{rem} \label{exchange}
	Notice that ${[a,b,c] \sim [c,-b,a]}$ via the matrix ${\begin{pmatrix} 0 & 1 \\ -1 & 0 \end{pmatrix} \in \slz{R}}$. Hence, $q$ is also equivalent to some form $[a',b',c']$ where $c'$ is coprime to $h$.
\end{rem}

\begin{prop} \label{coset}
	Let ${\Delta \in R \setminus \set{0}}$ and let ${q \in \cltwdp}$. Denote by $H_q$ its set of values in $\left(\faktor{R}{\Delta R}\right)^{\times}$, and by $H_0$ the set of values of the principal form class. Then $H_q$ is a coset of $H_0$ in $\left(\faktor{R}{\Delta R}\right)^{\times}$. More precisely, if $[a,b,c]$ is a representative of $q$ with ${a \neq 0}$ coprime to $\Delta$, then ${H_q = a^{-1}H_0}$.
\end{prop}

\begin{proof}
	We shall prove the last part of the statement, from which the result follows. As seen in Lemma~\ref{coprime}, there exists a representative $[a,b,c]$ of the class $q$ of quadratic forms such that $a$ is nonzero and coprime to $\Delta$. Let ${x,y \in R}$, let ${\pi \in R}$ such that ${\Delta \equiv \pi^2 \pmod{4R}}$, then
	\begin{align*}
	ax^2 + bxy + cy^2 & \equiv a^{-1}\left(ax+\frac{b-\pi}{2}y\right)^2 + \pi xy + \left(c-a^{-1}\left(\frac{b-\pi}{2}\right)^2\right)y^2 \pmod{\Delta} \\
	& \equiv a^{-1}\left(\left(ax+\frac{b-\pi}{2}y\right)^2 + \pi \left(ax+\frac{b-\pi}{2}y\right)y\right) \\
	& \hspace{4cm} + \left(-a^{-1}\pi\frac{b-\pi}{2} + c-a^{-1}\left(\frac{b-\pi}{2}\right)^2\right)y^2 \pmod{\Delta}.
	\end{align*}
	We compute the $y^2$-term: we have
	\begin{align*}
	-a^{-1}\pi\frac{b-\pi}{2} + c-a^{-1}\left(\frac{b-\pi}{2}\right)^2 & \equiv \frac{-2b\pi+2\pi^2 + 4ac-b^2+2b\pi-\pi^2}{4a}\pmod{\Delta} \\
	& \equiv \frac{\pi^2-\Delta}{4a}\pmod{\Delta}.
	\end{align*}
	Thus, we have
	$$ax^2 + bxy + cy^2 \equiv a^{-1}\left(X^2+\pi XY - \frac{\Delta-\pi^2}{4}Y^2\right)$$
	where ${X := ax+\frac{b-\pi}{2}y}$ and ${Y := y}$.
	
	This being true for all ${x,y \in R}$, we infer that ${H_q \subseteq a^{-1}H_0}$. For the reverse inclusion, notice that ${\begin{pmatrix} ax+\frac{b-\pi}{2}y \\ y \end{pmatrix} = \begin{pmatrix} a & \frac{b-\pi}{2} \\ 0 & 1 \end{pmatrix} \begin{pmatrix} x \\ y \end{pmatrix}}$, and the matrix $\begin{pmatrix} a & \frac{b-\pi}{2} \\ 0 & 1 \end{pmatrix}$ is invertible in $\left(\faktor{R}{\Delta R}\right)^{\times}$. Hence, for all ${X,Y \in R}$, we have
	$$a^{-1}\left[1,\pi,-\frac{\Delta-\pi^2}{4}\right](X,Y) \equiv [a,b,c](a^{-1}X + a^{-1}\frac{\pi-b}{2}Y,Y) \pmod{\Delta}.$$
	
	Thus, ${H_q = a^{-1}H_0}$, as desired.
\end{proof}

Finally, we can construct our desired group homomorphism.

\begin{theo} \label{groupmorphism}
	Let $R$ be a PID of characteristic different from $2$, let ${\Delta \in R}$ be a nonzero discriminant, and let $H_0$ be the set of values of the principal form class in $\left(\faktor{R}{\Delta R}\right)^{\times}$. The map
	$$\psi \colon \cltwdp \longrightarrow \faktor{\left(\faktor{R}{\Delta R}\right)^{\times}}{H_0}$$
	sending the class of a quadratic form to its set of values in $\left(\faktor{R}{\Delta R}\right)^{\times}$ modulo $H_0$ is well-defined and is a group homomorphism.
	
	Its kernel contains the squares, hence it factors through the map
	$$\Psi \colon \faktor{\cltwdp}{\cltwdp^{\square}} \longrightarrow \faktor{\left(\faktor{R}{\Delta R}\right)^{\times}}{H_0},$$
	where $\cltwdp^{\square}$ denotes the subgroup of squares of $\cltwdp$. We call $\Psi$ the \emph{Genus map}.
\end{theo}

\begin{proof}
	It follows from Proposition~\ref{coset} that $\psi$ is well-defined, and if ${q = [a,b,c] \in \cltwdp}$ with ${a \neq 0}$ coprime to $\Delta$, then ${\psi(q) = a^{-1}H_0}$.
	
	Let ${q_1, q_2 \in \cltwdp}$, we need to check that ${\psi(q_1 \ast q_2) = \psi(q_1) \psi(q_2)}$. Write ${q_1 = [a_1,b_1,c_1]}$ and ${q_2 = [a_2,b_2,c_2]}$. As seen in Lemma~\ref{coprime}, we may suppose that $a_1$ is coprime to $\Delta$, and that $a_2$ is coprime to $a_1\Delta$. Hence, ${\psi(q_1) = a_1^{-1}H_0}$ and ${\psi(q_2) = a_2^{-1}H_0}$. By Proposition~\ref{compo}, $a_1$ and $a_2$ being coprime, there exist ${b,c \in R}$ such that ${q_1 \ast q_2 = [a_1a_2,b,c]}$, leading to ${\psi(q_1 \ast q_2) = a_1^{-1}a_2^{-1}H_0 = \psi(q_1)\psi(q_2)}$, hence $\psi$ is a homomorphism.
	
	In view of Proposition~\ref{H0}, the group $H_0$ contains the subgroup of squares of $\left(\faktor{R}{\Delta R}\right)^{\times}$. Therefore, every square in $\cltwdp$ has image in $H_0$, hence ${\cltwdp^{\square} \subseteq \ker(\psi)}$. Thus, taking the quotient gives the desired group homomorphism $\Psi$.
\end{proof}

\begin{rem}
	Genus Theory as described in Theorem~\ref{groupmorphism} does not cover the case $\Delta=0$. However, in that case, the associated quadratic algebra $\faktor{R[\omega]}{\left\langle\omega^2\right\rangle}$ is degenerate, and the class group $\cl^{tw}_R(0)$ is trivial. Let us show this last point: if ${q = [a,b,c] \in \cl^{tw}_R(0)}$, then ${b^2=4ac}$. Write ${a = a_0^2\tilde{a}}$ and ${c = c_0^2\tilde{c}}$ for some square-free $\tilde{a},\tilde{c}$. Since ${b^2 = 4a_0^2c_0^2\tilde{a}\tilde{c}}$, and because $R$ is integrally closed, $\tilde{a}\tilde{c}$ must be a square. Hence, ${\tilde{c} = \varepsilon \tilde{a}}$ for some unit $\varepsilon$, since $\tilde{a}$ and $\tilde{c}$ are squarefree. Likewise, $\varepsilon$ must be a square, say ${\varepsilon = \nu^2}$, and we finally get ${[a,b,c] = [a_0^2\tilde{a}, \pm 2a_0c_0\tilde{a}\nu,c_0^2\nu^2\tilde{a}]}$. Now, let ${r,s \in R}$ be such that ${a_0r-(\pm c_0)s = \tilde{a}^{-1}}$, then ${\begin{pmatrix} a_0 & \pm c_0\nu \\ s & r \end{pmatrix} \cdot [1,0,0] = [a,b,c]}$, and our quadratic form is in the principal form class.
\end{rem}

\begin{rem}
	There are at least two other cases when the Genus map $\Psi$ from Theorem~\ref{groupmorphism} is trivial:
	\begin{enumerate}
		\item[$\bullet$] if ${\Delta \in R^{\times}}$, then the codomain of $\Psi$ is trivial;		
		\item[$\bullet$] if $\cltwdp$ is finite of odd order, then ${\cltwdp = \cltwdp^{\square}}$ and the domain of $\Psi$ is trivial.
	\end{enumerate}
\end{rem}

\begin{defi}
	The \emph{principal genus} is the kernel of the map $\psi$ defined in Theorem~\ref{groupmorphism}.
\end{defi}

Thus, in order to check that a given class $q$ of quadratic forms is non-trivial, it is sufficient to prove that $q$ is not in the principal genus. However, it is is not a necessary condition, as shown in the following example.

\begin{ex}
	Set ${R = \Z}$ and ${\Delta = -56}$. The quadratic form ${q_0 = [1,0,14]}$ is a representative of the principal form class, and the principal genus is ${\psi(q_0) = \pm\set{1,9,15,23,25,39}}$. Notice by the way that ${\pm\left(\faktor{\Z}{56\Z}\right)^{\times\square} \subsetneq \psi(q_0)}$. Let ${q_1 = [2,0,7]}$. Since ${q_1(1,1) = 9 = q_0(3,0)}$, and because sets of values in $\left(\faktor{\Z}{56\Z}\right)^{\times}$ of quadratic forms are either disjoint or equal as cosets of $\psi(q_0)$, we deduce that $q_1$ also lies in the principal genus. On the other hand, they are not equivalent since the equation ${x^2+14y^2 = \pm 2}$ has no solution ${(x,y) \in \Z^2}$.
\end{ex}

\subsection{The case $R = \Z$, $\Delta < 0$} \label{subsecrz}

We compare our map $\Psi$ from Theorem~\ref{groupmorphism} with Cox's exposition of its construction over $\Z$, for negative discriminants \cite[\S 3.B]{Cox2}. The classical construction of the group of classes of primitive binary quadratic forms of given discriminant ${\Delta < 0}$ rather uses $\SL_2(\Z)$ equivalence classes instead of the $\gltw$-ones. To achieve this, one notices that the $\gltw$-equivalence class of a given quadratic form $[a,b,c]$ consists in the union of the $\SL_2(\Z)$-equivalence classes of $[a,b,c]$ and ${\begin{pmatrix} 1 & 0 \\ 0 & -1 \end{pmatrix} \cdot [a,b,c] = [-a,b,-c]}$. In other words, there is a kind of duplication of equivalence classes when one goes from $\gltw$ to $\SL_2(\Z)$, splitting positive definite and negative definite quadratic forms. Thus, when one considers $\SL_2(\Z)$-classes, one first removes negative definite quadratic forms, so that there are as many $\SL_2(\Z)$-equivalence classes of \textbf{positive definite} quadratic forms as $\gltw$-equivalence classes of quadratic forms (positive or negative).

Given some negative integer ${\Delta \equiv 0}$ or ${1 \pmod{4}}$ (the only possible cases over $\Z$), we denote by $\clzsl(\Delta)$ the group of $\SL_2(\Z)$-equivalence classes of primitive positive definite binary quadratic forms of discriminant $\Delta$. According to the above discussion, the natural group homomorphism ${\clzsl(\Delta) \longrightarrow \clztw(\Delta)}$ (which assigns to an $\SL_2(\Z)$-equivalence class the \linebreak$\gltw$-equivalence class it spans) is an isomorphism.

The values in $\left(\faktor{\Z}{\Delta\Z}\right)^{\times}$ taken by a primitive quadratic form of discriminant $\Delta$ are related to the kernel of the Dirichlet character ${\chi \colon \left(\faktor{\Z}{\Delta\Z}\right)^{\times} \longrightarrow \set{\pm 1}}$ defined for all odd primes $p$ not dividing $\Delta$ by ${\chi(p) := \genfrac(){}{0}{\Delta}{p}}$, where $\genfrac(){}{0}{.}{p}$ is the Legendre symbol. See \cite[Lemma~1.14]{Cox2} for a detailed exposition.

Classically, the main result of Genus Theory over $\Z$ is the following: if ${\Delta \equiv 0,1 \pmod{4}}$ is a negative integer, then we have an isomorphism of abelian groups
$$\faktor{\clzsl(\Delta)}{\clzsl(\Delta)^{\square}} \overset{\sim}{\longrightarrow} \faktor{\ker(\chi)}{\val_{\Delta}(q_0)}$$
where $\val_{\Delta}(q_0)$ is the subgroup of values taken by the principal form $q_0$ in $\left(\faktor{\Z}{\Delta\Z}\right)^{\times}$. See \cite[Theorem~3.15]{Cox2}.

In order to compare with our version of Genus Theory, we consider the following diagram

\begin{equation}
\begin{tikzcd} \label{diagcoxlink}
\faktor{\clzsl(\Delta)}{\clzsl(\Delta)^{\square}}
\arrow[r,"\sim "] \arrow[d,"\sim"' labl]
& \faktor{\ker(\chi)}{\val_{\Delta}(q_0)} \arrow[d,"\varphi"] \\
\faktor{\clztw(\Delta)}{\clztw(\Delta)^{\square}} \arrow[r,"\Psi_{\Z}"] &
\faktor{\left(\faktor{\Z}{\Delta\Z}\right)^{\times}}{\pm \val_{\Delta}(q_0)}
\end{tikzcd}
\end{equation}
where $\Psi_{\Z}$ is the homomorphism from Theorem~\ref{groupmorphism} with ${R = \Z}$, and $\varphi$ is the one induced by the inclusion ${\ker(\chi) \hookrightarrow \left(\faktor{\Z}{\Delta\Z}\right)^{\times}}$. Diagram (\ref{diagcoxlink}) is commutative since a given class of quadratic forms ${q = [a,b,c]}$ in the top left corner with $a$ coprime to $\Delta$ (possible by Lemma~\ref{coprime} and Remark~\ref{remsl2}) has image ${\pm a^{-1}\val_{\Delta}(q_0)}$ in the bottom right corner, whatever the chosen path.

\begin{prop} \label{propcasez}
	Let ${\Delta \equiv 0,1 \pmod{4}}$ be a negative integer. Then the maps $\varphi$ and $\Psi_{\Z}$ from Diagram (\ref{diagcoxlink}) are isomorphisms of abelian groups.
\end{prop}

\begin{proof}
	By commutativity of Diagram (\ref{diagcoxlink}), it is enough to show that $\varphi$ is an isomorphism. We start by showing that the source and the target of $\varphi$ have the same size. On the one hand, $\left(\faktor{\Z}{\Delta\Z}\right)^{\times}$ is the disjoint union of the fibres of $\chi$, hence ${\abs{\left(\faktor{\Z}{\Delta\Z}\right)^{\times}} = 2\abs{\ker(\chi)}}$. On the other hand, ${-1 \notin \val_{\Delta}(q_0)}$ since $\val_{\Delta}(q_0)$ is a subgroup of $\ker(\chi)$, but ${\chi(-1) = -1}$ according to \cite[Lemma~1.14]{Cox2}, since $\Delta$ is negative. We thus find that the domain and codomain of $\varphi$ have the same size.
	
	Let us check that $\varphi$ is injective. If ${\alpha \in \ker(\chi)}$, then $-\alpha$ is not in $\ker(\chi)$, hence not in $\val_{\Delta}(q_0)$. This implies that if ${\varphi(\alpha) \in \pm \val_{\Delta}(q_0)}$, then $\alpha$ must be in $\val_{\Delta}(q_0)$. So $\varphi$ is injective, hence bijective by an argument of cardinality.
\end{proof}

In particular, when ${R = \Z}$ and ${\Delta \equiv 0,1 \pmod{4}}$ is negative, the principal genus consists precisely in the subgroup of squares $\clztw(\Delta)^{\square}$.

\subsection{The case $R = \K[X], \Delta = 4f$, and the link with $2$-descent} \label{subsec2descente}

Let $\K$ be a field of characteristic $0$, and let ${f \in \K[X]}$ be a square-free monic polynomial of odd degree ${2g+1}$ with ${g \geq 1}$. We now consider Genus Theory over ${R = \K[X]}$ with discriminant of the specific form ${\Delta = 4f}$, and we compare it to the $2$-descent map of some hyperelliptic curve. Notice that we can write the middle coefficient of a quadratic form as $2b$ instead of $b$; this is consistent with the fact that ${\Delta = 4f}$, and it enables us to avoid fractions while doing computations. Furthermore, if ${q = [a,2b,c] \in \clkx}$, then ${a \neq 0}$ since $f$ has odd degree.

For the following, denote ${L := \faktor{\K[X]}{\left\langle f(X) \right\rangle}}$. Since $2$ is a unit in $\K[X]$, the set $H_0$ of values taken by the principal form class in $L^{\times}$ is just $\K^{\times}L^{\times \square}$ (Proposition~\ref{H0}). In that context, the group homomorphism from Theorem~\ref{groupmorphism} can be written as
\begin{equation} \label{mappsi}
\Psi \colon \left \{ \begin{aligned} \faktor{\clkx}{\clkx^{\square}} & \longrightarrow \faktor{L^{\times}}{\K^{\times}L^{\times \square}} \\ \Big[[a,2b,c]\Big] \text{~with~} \gcd(a,f) = 1 & \longmapsto a^{-1}\K^{\times}L^{\times \square} \end{aligned} \right..
\end{equation}

Recall that ${\clkx \simeq \pic\left(\faktor{\K[X,Y]}{\left\langle Y^2 - f(X) \right\rangle}\right)}$ by Proposition~\ref{propbij}. In order to make the link with $2$-descent, note that $\spec\left(\faktor{\K[X,Y]}{\left\langle Y^2 - f(X) \right\rangle}\right)$ is a degree $2$-cover of $\A^1_{\K}$. As a smooth affine curve, it can be uniquely completed into a smooth projective curve $\mathcal{C}$. The curve $\mathcal{C}$ is a \emph{hyperelliptic curve over $\K$}, that is, a smooth projective geometrically connected $\K$-curve of genus ${g \geq 1}$, endowed with a degree $2$ map ${\mathcal{C} \longrightarrow \PP_{\K}^1}$.

\begin{rem} \label{remodddegree}
	Since $f$ has odd degree, the hyperelliptic curve $\mathcal{C}$ has a $\K$-rational Weierstrass point $\infty$ lying above the point at infinity of $\PP^1_{\K}$.
	
	Conversely, given a hyperelliptic curve over $\K$ with a rational Weierstrass point, we can shift it above the point at infinity. Then it is a standard fact that the curve deprived of this point can be described by an affine equation ${Y^2=f(X)}$ where ${f \in \K[X]}$ is square-free and monic of odd degree.
\end{rem}

Denote by $J$ the Jacobian variety of $\mathcal{C}$. Let ${D = \displaystyle{\sum_{i=1}^r n_i\Big((x_i,y_i) - \infty\Big)}}$ be a degree $0$ divisor on $\mathcal{C}$, and assume that none of the $(x_i,y_i)$ is a Weierstrass point. Set ${\overline{L} := \faktor{\overline{\K}[X]}{\left\langle f(X) \right\rangle}}$ where $\overline{\K}$ is the algebraic closure of $\K$. We define ${\lambda(D) := \displaystyle{\prod_{i=1}^r (x_i-X)^{n_i} \in \overline{L}^{\times}}}$, and it may be shown that this induces a group homomorphism
$$\lambda \colon \faktor{\jk}{2\jk} \longrightarrow \faktor{L^{\times}}{L^{\times \square}},$$
which we refer to as the \emph{$2$-descent map} \cite[Lemmas~2.1 and 2.2]{Sch}.

Now, let us describe the isomorphism between ${\clkx \simeq \pic\left(\faktor{\K[X,Y]}{\left\langle Y^2 - f(X) \right\rangle}\right)}$ and ${\jk = \pic^0(\mathcal{C}) \simeq \pic(\mathcal{C}\setminus \set{\infty})}$. Let ${[a,2b,c] \in \clkx}$. We know from bijection~\eqref{bij} that the class of the quadratic form $[a,2b,c]$ correponds to the ideal class ${\left\langle a, Y-b \right\rangle}$ in $\pic\left(\faktor{\K[X,Y]}{\left\langle Y^2 - f(X) \right\rangle}\right)$. This induces a Weil divisor on ${\mathcal{C}\setminus \set{\infty}}$ defined as the vanishing locus of ${\left\langle a, Y-b \right\rangle}$, which we denote by ${\dv(a) \cap \dv(Y-b)}$. Since our hyperelliptic curve is smooth, this indeed corresponds to a Cartier divisor on ${\mathcal{C}\setminus \set{\infty}}$. We associate to it a degree $0$ divisor on $\mathcal{C}$, that is, a point in $J(\K)$, by removing a suitable multiple of $\infty$. We thus obtain a group isomorphism

$$\mathcal{D} \colon \left \{ \begin{aligned} \clkx & \longrightarrow \jk \\ \Big[q = [a,2b,c]\Big] & \longmapsto \dv(a) \cap \dv(Y-b)-\deg(a)\infty \end{aligned} \right..$$
Notice that ${\dv(a) \cap \dv(Y-b) = \underset{i=1}{\overset{\deg(a)}{\sum}} (x_i,y_i)}$ where ${\set{x_1,\ldots,x_{\deg(a)}} \subset \overline{\K}}$ is the set of roots of $a$, and ${b(x_i)=y_i}$ for all $i$.

\begin{rem}
	Mumford parametrized divisor classes in ${J(\K) = \pic^0(\mathcal{C})}$ by triples of polynomials, and it has been well-known that they correspond to the coefficients of the associated quadratic forms as above. He further proved that a given divisor class has a unique reduced representative, whose associated triple of polynomials satisfies some bounds on their degrees \cite[Proposition~1.2 and page~3.29]{Mumford}. Applying a reduction algorithm based on Euclidean division allows to fully recover his parametrization.
\end{rem}

We still denote by $\mathcal{D}$ the induced map ${\faktor{\clkx}{\clkx^{\square}} \overset{\sim}{\longrightarrow} \faktor{\jk}{2\jk}}$, by abuse of notations. In order to compare Genus Theory and $2$-descent, we reproduce Diagram (\ref{diag2descintro}) from the introduction:
\begin{equation}
\begin{tikzcd} \label{diag2desc}
\faktor{\clkx}{\clkx^{\square}}
\arrow[r,"\sim ", "\mathcal{D}"'] \arrow[d,"\Psi "']
& \faktor{\jk}{2\jk} \arrow[d,"\lambda"] \\
\faktor{L^{\times}}{\K^{\times}L^{\times \square}} &
\faktor{L^{\times}}{L^{\times \square}}
\arrow[l,"pr "]
\end{tikzcd}
\end{equation}
where $pr$ is the natural projection, sending $\alpha L^{\times \square}$ to $\alpha \K^{\times}L^{\times \square}$ for all ${\alpha \in L^{\times}}$. Our first goal is to prove that this diagram is commutative. Then, we will derive the injectivity of $\Psi$ from the injectivity of $\lambda$.

\begin{prop} \label{explx-t}
	Let ${f \in \K[X]}$ be a square-free monic polynomial of odd degree ${2g+1}$ with ${g \geq 1}$. Let ${q \in \clkx}$, and let $[a,2b,c]$ be a representative of $q$ with $a$ coprime to $f$ (possible by Lemma~\ref{coprime}). Then we have
	$$\lambda(\mathcal{D}(q)) = \frac{(-1)^{\deg(a)}}{\lc(a)}aL^{\times \square},$$
	where $\lc(a)$ denotes the leading coefficient of $a$.
	
	In particular, Diagram (\ref{diag2desc}) is commutative.
\end{prop}

\begin{proof}
	Write ${\dv(a) \cap \dv(Y-b) = \underset{i=1}{\overset{\deg(a)}{\sum}} (x_i,y_i)}$ with ${a(x_i)=0}$ and ${b(x_i)=y_i}$. If ${y_i = 0}$, then $x_i$ is both a root of $a$ and $b$, hence a root of ${b^2-ac = f}$. Since $a$ and $f$ are coprime by assumption, this cannot happen, and $y_i$ must be nonzero.
	
	First, assume that $a$ is an irreducible polynomial, and denote by $d$ its degree. Since the points ${(x_i,y_i)}$ appearing in $\mathcal{D}(q)$ are such that ${a(x_i) = 0}$, we have
	$$\dv(a) \cap \dv(Y-b) - d\infty = \sum_{i=1}^d (x_i,y_i) - d\infty = \sum_{i=1}^d \sigma_i((x_1,y_1)) - d\infty,$$
	where the $\sigma_i((x_1,y_1))$ are all the conjugates in $\overline{\K}$ over $\K$ of ${(x_1,y_1)}$. Applying \cite[Lemma~2.2]{Sch}, we get
	\begin{equation*}
	\lambda(\mathcal{D}(q)) =  \left(\prod_{i=1}^d (x_i - X)\right)L^{\times \square} = \frac{(-1)^d}{\lc(a)} a L^{\times \square}.
	\end{equation*}
	
	In the general case, if ${\displaystyle{a = \prod_i a_i}}$ is the decomposition of $a$ into irreducible factors, then
	$$\dv(a) \cap \dv(Y-b) - \deg(a)\infty = \sum_i \Big(\dv(a_i) \cap \dv(Y-b) - \deg(a_i)\infty\Big),$$
	hence the result follows from the irreducible case.
	
	To conclude, the diagram is commutative since
	$$pr \circ \lambda \circ \mathcal{D}(q) = pr\left(\frac{(-1)^{\deg(a)}}{\lc(a)}aL^{\times \square}\right) = a\K^{\times}L^{\times\square} = \Psi(q),$$
	hence the result
\end{proof}

Thus, the commutativity of Diagram (\ref{diag2desc}) reveals a strong relation between Genus Theory over $\K[X]$ and $2$-descent on hyperelliptic curves over $\K$, at least when $\deg(f)$ is odd.

\begin{theo} \label{inj}
	Let ${f \in \K[X]}$ be a square-free monic polynomial of odd degree ${2g+1}$ with ${g \geq 1}$. Then the Genus map $\Psi$ from (\ref{mappsi}) is injective.
\end{theo}

\begin{proof}
	Since $\mathcal{D}$ is an isomorphism in Diagram (\ref{diag2desc}), it is enough to show that ${pr \circ \lambda}$ is injective. Let ${D \in \ker(pr \circ \lambda)}$. Then ${\lambda(D) = \K^{\times}L^{\times\square}}$, and we can choose a representative $\varepsilon a^2$ of ${\lambda(D)}$ for some ${\varepsilon \in \K^{\times}}$ and some ${a \in L^{\times}}$.
	
	According to \cite[Theorem~1.1]{Sch}, the quantity $\varepsilon a^2$ is in the kernel of the \emph{norm} ${N \colon \faktor{L^{\times}}{L^{\times\square}} \longrightarrow \faktor{\K^{\times}}{\K^{\times\square}}}$, which is the restriction of the usual norm ${\faktor{\overline{\K}[X]}{\left\langle f(X) \right\rangle} \simeq \overline{\K}^{2g+1} \longrightarrow \overline{\K}}$ given by ${(\alpha_1,\ldots,\alpha_{2g+1}) \mapsto \underset{i=1}{\overset{2g+1}{\prod}} \alpha_i}$. We have ${N(\varepsilon a^2) = N(\varepsilon) N(a)^2 = N(\varepsilon)}$, since squares are in the kernel of $N$. As ${\varepsilon \in \K}$, we get ${N(\varepsilon) = \varepsilon^{2g+1}}$, which is in the same class as $\varepsilon$ modulo the squares. From all of this we deduce that $\varepsilon$ must be a square, and ${\lambda(D) = L^{\times\square}}$, meaning that ${D \in \ker(\lambda)}$. But $\lambda$ is injective by \cite[Theorem~1.2]{Sch}, hence ${D = 0}$ and ${pr \circ \lambda}$ is injective, as desired.
\end{proof}

\begin{rem}
	The fact that $f$ has odd degree is crucial in the proof of Theorem~\ref{inj}. It also plays an important role in the construction of Diagram (\ref{diag2desc}): if $f$ has even degree, then the corresponding hyperelliptic curve $\mathcal{C}$ has two points $\infty_+$ and $\infty_-$ at infinity, and it is not clear how to relate ${\pic(\mathcal{C}\setminus\set{\infty_+,\infty_-})}$ to $\pic^0(\mathcal{C})$. On the other hand, when $\deg(f)$ is even, the $2$-descent map $\lambda$ is slightly different; in particular, its codomain is no longer $\faktor{L^{\times}}{L^{\times \square}}$ but $\faktor{L^{\times}}{\K^{\times}L^{\times \square}}$. Furthermore, it may be non injective: according to \cite[Proposition~5]{Poonen}, if ${f \in \Q[X]}$ has degree $6$, is irreducible and has Galois group $S_6$, then ${\ker(\lambda)}$ has order $2$. It would be interesting to see if this affects the possible injectivity of the Genus map in that context.
\end{rem}

\section{Non-trivial specializations in families of class groups of quadratic fields extensions} \label{secspecial}

\subsection{Description of the problem} \label{subsecmotivspecialization}

From now onwards, we consider a hyperelliptic curve $\mathcal{C}$ of genus ${g \geq 1}$ over $\K$ as in Subsection~\ref{subsec2descente}, but this time $\K$ is a number field. We assume that $\mathcal{C}$ has a rational Weierstrass point. As already mentioned in Remark~\ref{remodddegree}, we shift this point above the point at infinity and we denote it by $\infty$. We then choose an affine equation of ${\mathcal{C} \setminus \set{\infty}}$ of the form ${Y^2 = f(X)}$ where $f$ is a square-free monic polynomial of degree ${2g+1}$. In this setting, we further assume without loss of generality that ${f \in \okx}$.

Let ${\mathcal{W} := \spec\left(\faktor{\ok[X,Y]}{\left\langle Y^2-f(X) \right\rangle}\right)}$. Then $\mathcal{W}$ is an affine $\ok$-scheme whose generic fibre is the curve ${\mathcal{C} \setminus \set{\infty}}$. Our motivation is the following. Given a non-trivial ideal class $I$ in $\pic\left(\mathcal{W}\right)$, can we find algebraic integers ${n \in \ok}$ such that the ``specialization'' of $I$ at ${X = n}$ is a non-trivial ideal class in $\pic\left(\faktor{\ok[Y]}{\left\langle Y^2-f(n)\right\rangle}\right)$\,? Or even better, in $\pic\left(\rint{\K\left(\sqrt{f(n)}\right)}\right)$, provided $f(n)$ is not a square\,?

\begin{ex}\label{exsieve}
	Let ${\K = \Q}$ and let ${f(X) = X^3-X+9 \in \Z[X]}$. Then $\mathcal{C}$ is an elliptic curve and ${I := \left\langle X, Y-3 \right\rangle}$ is an ideal of $\faktor{\Z[X,Y]}{\left\langle Y^2-X^3+X-9 \right\rangle}$. This ideal $I$ is not principal. For which ${n \in \Z}$ can we say that the specialized ideal ${I_n := \left\langle n, Y-3 \right\rangle}$ is not principal in $\faktor{\Z[Y]}{\left\langle Y^2-n^3+n-9 \right\rangle}$\,?
	
	This problem can be restated in terms of quadratic forms, using Proposition~\ref{propbij}. For which ${n \in \Z}$ can we say that the specialized integral quadratic form ${q(n) := [n,6,-n^2+1]}$ of discriminant ${4n^3-4n+36}$ is not $\gltw$-equivalent to the principal form ${q_0(n) := [1,0,-n^3+n-9]}$\,?
	
	The theory of reduced quadratic forms gives algorithms to compute whether a given quadratic form (of nonsquare discriminant) is equivalent to the principal form or not (see \textit{e.g.} \cite[\S 5.4.2, 5.6.1]{Cohen}). Using them, we look for simple patterns in the distribution of non-trivial specializations. Figure~\ref{sieve} summarizes the data obtained for integers $n$ from $1$ to $100$ (all corresponding to the case of positive discriminant): a cell in the $i^{\text{th}}$-row and $j^{\text{th}}$-column corresponds to the integer ${n = j + 5i}$ for ${i = 0,\ldots,19}$ and ${j = 1,\ldots,5}$. The cell matching to the integer $n$ is red and hatched when the corresponding quadratic form ${q(n) = [n,6,-n^2+1]}$ is equivalent to the principal form $q_0(n)$. It is yellow when $q(n)$ is not equivalent to $q_0(n)$. It is blue and gridded when $f(n)$ is a square, in which case the quadratic algebra $\faktor{\Z[Y]}{\left\langle Y^2 - f(n) \right\rangle}$ is degenerate.
	
	\begin{figure}[H]
		\centering
		\includegraphics[scale=0.5]{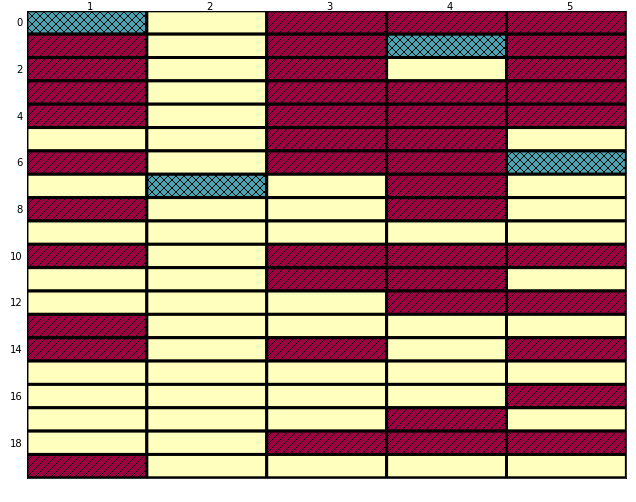}
		\caption{Sieve of non-trivial specializations of $q(n)$ for ${n \in \llbracket 1,100 \rrbracket}$ \label{sieve}}
	\end{figure}
	
	We observe that the second column in Figure~\ref{sieve} does not contain any red cell. This leads us to conjecture that when ${n \equiv 2 \pmod{5}}$, then $q(n)$ is never equivalent to $q_0(n)$. Let us show this: when ${n \equiv 2 \pmod{5}}$, we have ${f(n) \equiv 2^3-2+9 \equiv 0 \pmod{5}}$, ${q(n) \equiv [2,1,-3] \pmod{5}}$, and ${q_0(n) \equiv [1,0,0] \pmod{5}}$. For a contradiction, if $q_0(n)$ were equivalent to $q(n)$, there would exist ${x,y \in \Z}$ such that ${x^2 - f(n)y^2 = \pm n}$. This implies that ${x^2 \equiv \pm 2 \pmod{5}}$, which is impossible.
	
	Thus, in this example, we have found an infinite family of non-trivial specializations of $q$, namely all the ${n \in \Z}$ such that ${n \equiv 2 \pmod{5}}$. Actually, there are other modular criteria: if ${n \equiv 2 \pmod{12}}$, or if ${n \equiv 32 \pmod{37}}$, for example, then again $q(n)$ is not equivalent to $q_0(n)$.
\end{ex}

Our goal is to prove that such congruence classes exist in general. As stated in Theorem~\ref{theocritmod}, Genus Theory gives a way to produce such classes. Although this process is constructive, and based on the same arguments as in Example~\ref{exsieve}, it relies on certain properties of the rings of integers of the fields generated by the roots of $f$. Providing effective congruence classes would certainly require some additional work.

Let us come back to the general case, over a number field $\K$. For technical purposes, we are led to invert a suitable set of places. For $\mathcal{S}$ a finite set of nonzero prime ideals of $\ok$, we consider the ring of $\mathcal{S}$-integers
$$\oks = \set{x \in \K \st \nu_{\mathfrak{p}}(x) \geq 0 ~\forall \mathfrak{p} \not\in \mathcal{S}},$$
where $\nu_{\mathfrak{p}}(x)$ is the $\mathfrak{p}$-adic valuation of $x$. Then, for a choice of $\mathcal{S}$ that we will make precise soon, we modify the problem as follows: given a non-trivial ideal class $I$ in $\pic\left(\faktor{\oks[X,Y]}{\left\langle Y^2-f(X) \right\rangle}\right)$, we look for ${n \in \oks}$ such that $I$ is non-trivial in $\pic\left(\faktor{\oks[Y]}{\left\langle Y^2-f(n)\right\rangle}\right)$ after specialization at ${X = n}$. 

The following Proposition tells us what set $\mathcal{S}$ we should consider in order to take advantage of Genus Theory over $\K[X]$.

\begin{prop} \label{propoksisok}
	Let ${f \in \ok[X]}$ be a square-free monic polynomial of odd degree at least $3$ and let $\mathcal{S}$ be a finite set of nonzero prime ideals of $\ok$. The restriction to the generic fibre induces a homomorphism of abelian groups
	$$\theta \colon \pic\left(\faktor{\oks[X,Y]}{\left\langle Y^2-f(X) \right\rangle}\right) \longrightarrow \pic\left(\faktor{\K[X,Y]}{\left\langle Y^2-f(X) \right\rangle}\right).$$
	Moreover,
	\begin{enumerate}
		\item[$\bullet$] if $\mathcal{S}$ contains all the prime ideals dividing $2\disc(f)$, then $\theta$ is surjective;
		\item[$\bullet$] if $\mathcal{S}$ is such that $\oks$ is a PID, then $\theta$ is injective. 
	\end{enumerate}
\end{prop}

\begin{proof}
	The map $\theta$ corresponds to the restriction of a given divisor on ${\ws = \spec\left(\faktor{\oks[X,Y]}{\left\langle Y^2-f(X) \right\rangle}\right)}$ to the generic fibre. The main steps of the proof are extracted from the first part of the proof of \cite[Lemma~2.4]{Jean}, which is stated over $\Z$, but directly extends to $\oks$.
	
	Given a divisor on the generic fibre, its scheme-theoretic closure on the integral model $\ws$ gives a Weil divisor on $\spec(\ws)$. When the prime ideals dividing $2\disc(f)$ are inverted, the Jacobian criterion shows that the affine $\spec(\oks)$-scheme $\ws$ is smooth over $\oks$, hence Weil and Cartier divisors coincide on $\ws$. This proves surjectivity of $\theta$ in that case.
	
	For injectivity, let ${D \in \ker(\theta)}$. Then ${\theta(D) = \dv_{\K}(h)}$ for some ${h \in \faktor{\K[X,Y]}{\left\langle Y^2-f(X) \right\rangle}}$. Since $\ws$ and ${\mathcal{C} \setminus \set{\infty}}$ have the same function field, we can consider $\dv(h)$ as a principal divisor over $\ws$. Thus, we see that $D$ and $\dv(h)$ have the same generic fibre, hence ${D - \dv(h)}$ is a vertical divisor. On the other hand, since $f$ is monic of odd degree, $f$ cannot be a square modulo any prime ideal of $\oks$, implying that the fibres of $\ws$ are irreducible. Therefore, vertical divisors on $\ws$ are sum of fibres. When the prime ideals dividing a chosen set of generators of $\pic(\oks)$ are inverted, $\oks$ is a PID and the fibres of $\ws$ are principal. If this happens, $D$ is principal, and $\theta$ is injective.
\end{proof}

Now, \textbf{we fix once and for all a finite set $\mathcal{S}$ of nonzero prime ideals of $\ok$, such that $\oks$ is a PID}.
%\mathcal{S} := \set{\mathfrak{p} \in \spec(\ok) \st \mathfrak{p} \text{~divides~} 2\disc(f)I_1\ldots I_r}.
%Then the homomorphism $\theta$ from Proposition~\ref{propoksisok} is an isomorphism.
Then the homomorphism $\theta$ from Proposition~\ref{propoksisok} is injective.

%\begin{rem}
%	Notice that in particular, $\mathcal{S}$ contains all the primes of bad reduction for the curve $\mathcal{C}$, and that the ring $\oks$ is a PID.
%\end{rem}

\begin{rem}\label{choice}
	In particular, we can choose $\mathcal{S}$ such that all the prime ideals dividing $2\disc(f)$ belong to $\mathcal{S}$. In that case, $\mathcal{S}$ contains all the primes of bad reduction for the curve $\mathcal{C}$. Furthermore, the homomorphism $\theta$ from Proposition~\ref{propoksisok} is then an isomorphism, and the link with Genus Theory over $\K[X]$ is strengthened. This is the kind of set to consider while addressing Agboola and Pappas' question, mentioned in the end of the Introduction. 
\end{rem}

\begin{rem}
	When the homomorphim $\theta$ in Proposition~\ref{propoksisok} is an isomorphism, Sivertsen and Soleng gave an algorithm to compute effectively its inverse (\cite[Lemma~3.2]{SivSol}).
\end{rem}

Since $\oks$ is a PID, every locally free $\oks[X]$-module of finite rank is free, according to \cite{Seshadri}. Therefore, every ideal class in $\pic\left(\faktor{\oks[X,Y]}{\left\langle Y^2-f(X) \right\rangle}\right)$ has a representative of the form ${\left\langle A(X), Y - B(X) \right\rangle}$ with ${A \neq 0}$, and corresponds to the class of the quadratic form ${q := [A(X),2B(X),\frac{B(X)^2-f(X)}{A(X)}]}$, by Proposition~\ref{propbij}. Thus, we deduce the quadratic form version of Proposition~\ref{propoksisok}.

\begin{cor} \label{coroksisokqf}
	Let ${f \in \ok[X]}$ be a square-free monic polynomial of odd degree at least $3$. When $\mathcal{S}$ is a finite set of nonzero prime ideals of $\ok$ such that $\oks$ is a PID, the injective homomorphism from Proposition~\ref{propoksisok} induces an injective homomorphism $$\cloksx \hookrightarrow \clkx.$$
\end{cor}

From now on we will mainly work with quadratic forms.

Let ${n \in \oks}$. Given an ideal class in $\pic\left(\faktor{\oks[X,Y]}{\left\langle Y^2-f(X) \right\rangle}\right)$, that is, given a class of quadratic forms in $\cloksx$, we obtain a class of quadratic forms in $\cloks$ by applying the evaluation homomorphism
\begin{equation} \label{eval}
ev_n \colon \cloksx \longrightarrow \cloks
\end{equation}
which sends the class of $[A,2B,C]$ to the class of ${[A(n),2B(n),C(n)]}$.

\begin{rem}
	When $f(n)$ is not a square, the target of (\ref{eval}) is the Picard group of $\oks\left[\sqrt{f(n)}\right]$. But when $f(n)$ is a square, possibly $0$, the quadratic algebra one obtains is no longer an integral domain.
\end{rem}

\subsection{Genus Theory gives a modular criterion}

In the following, $f$ still denotes a square-free monic polynomial of odd degree at least $3$ with coefficients in $\ok$, and we let $\mathcal{S}$ be a finite set of nonzero prime ideals of $\ok$ such that $\oks$ is a PID. Recall that the Genus Theory presented in Section~\ref{secgentheory} requires to work over a principal ideal domain. Thus, we will often use the injective homomorphism from Corollary~\ref{coroksisokqf} which makes a class of quadratic forms in $\cloksx$ correspond to a class of quadratic forms in $\clkx$. On the other hand, once a given class of quadratic forms ${q \in \cloksx}$ is evaluated at some ${n \in \oks}$, we get a class of quadratic forms over $\oks$ which is a PID, hence Genus Theory directly applies.

Let ${L := \faktor{\K[X]}{\left\langle f(X) \right\rangle}}$. Then, for all ${n \in \oks}$ such that ${f(n) \neq 0}$, we have the following landscape, where ${M_n := \faktor{\oks}{f(n)\oks}}$~:
\begin{equation} \label{diaggenuseval}
\begin{tikzcd}
& \faktor{\cloksx}{\cloksx^{\square}} \arrow[ld,hookrightarrow] \arrow[dr, "\overline{ev_n}"] & \\
\faktor{\clkx}{\clkx^{\square}} \arrow[d, hookrightarrow, "\Psi"] & & \faktor{\cloks}{\cloks^{\square}} \arrow[d, "\psi_n"] \\
\faktor{L^{\times}}{\K^{\times}L^{\times \square}} & & \faktor{M_n^{\times}}{\oks^{\times}M_n^{\times \square}}
\end{tikzcd}
\end{equation}
The two vertical maps are the homomorphisms corresponding to Genus Theory. The left one, $\Psi$, has already been defined in (\ref{mappsi}), and is injective by Theorem~\ref{inj}, whereas $\psi_n$ is obtained when we set ${R := \oks}$ and ${\Delta := 4f(n)}$ in Theorem~\ref{groupmorphism}. Beside this, the map $\overline{ev_n}$ is the one induced on the quotients by $ev_n$, defined in (\ref{eval}).

Given a class $q$ of quadratic forms over $\oks[X]$ which is not in the principal genus, Diagram~(\ref{diaggenuseval}) outlines what will be our strategy to find some ${n \in \oks}$ such that ${ev_n(q) \in \cloks}$ is non-trivial. We will use the map $\Psi$ to get some information, namely the fact that certain quantities are not squares (see Corollary~\ref{rootnotsquareint}). Linking those quantities to ${\psi_n \circ \overline{ev_n}(q)}$ will give us clues about what ${n \in \oks}$ we should choose (see PGS Theorem~\ref{theocritmod}).

\begin{rem}
	Our approach enables us to do a little bit more than finding non-trivial specializations of a given class of quadratic forms. Indeed, since the genus maps are group homomorphisms, we can directly extend our considerations to the question of finding non-equivalent specializations in $\cloks$ of two given distinct classes of quadratic forms ${q,q' \in \cloksx}$. When $q$ and $q'$ are not in the same genus over $\K[X]$, that is, when they do not have the same image through $\Psi$, the arguments of this paper apply to ${q' \ast q^{-1}}$, implying that their specializations are non-equivalent.
\end{rem}

%We know from Corollary~\ref{coroksisokqf} that every quadratic form over $\K[X]$ is $\gltw$-equivalent to one with coefficients in $\oks[X]$.
The following result is an integral version of Lemma~\ref{coprime}.

\begin{lem} \label{intrep}
	Every class of quadratic forms in $\cloksx$ has a representative of the form $[A,2B,C]$ where $A$ is coprime to $f$ in $\K[X]$.
\end{lem}

\begin{proof}
	Let ${q = [A_0,2B_0,C_0] \in \cloksx}$. By Remark~\ref{remcompute}, every quadratic form equivalent to $q$ has its first coefficient of the form ${\varepsilon q(\alpha,\gamma) = \varepsilon(A_0\alpha^2 + 2B_0\alpha\gamma + C_0\gamma^2)}$ for some ${\varepsilon \in \oks^{\times}}$ and some ${\alpha,\gamma \in \oks[X]}$ such that ${\left\langle \alpha, \gamma \right\rangle = \oks[X]}$. Actually, restricting our search to ${\varepsilon = 1}$, ${\gamma = 1}$ and ${\alpha \in \oks}$ will be enough for our purpose.
	
	Let $\K(f)$ be the splitting field of $f$. For all ${\alpha \in \oks}$, the polynomial ${q(\alpha,1)}$ is coprime to $f$ in $\K[X]$ if and only if we have ${q(\alpha,1)(\rho) \neq 0}$ for all $\rho$ roots of $f$ in $\K(f)$. In other words, we look for ${\alpha \in \oks}$ such that
	$$P(\rho) := A_0(\rho)\alpha^2 + 2B_0(\rho)\alpha + C_0(\rho) \neq 0 ~~~~ \text{~for~all~} \rho \text{~root~of~} f.$$
	Fix some root $\rho$ of $f$. Since the quadratic form $[A_0,2B_0,C_0]$ is primitive, we cannot have ${A_0(\rho)=B_0(\rho)=C_0(\rho)=0}$ at the same time, hence $P(\rho)$ is a nonzero polynomial in the variable $\alpha$, and has at most $2$ roots in $\oks$. Doing this for all $\rho$, we have at most $2\deg(f)$ values of $\alpha$ to avoid. Since $\K$ is infinite, we can take $\alpha \in \oks$ not in the set of those values. Then the quadratic form ${[A,2B,C] := \begin{pmatrix} \alpha & -1 \\ 1 & 0 \end{pmatrix} \cdot [A_0,2B_0,C_0]}$ is another representative of our class $q$, with $A$ coprime to $f$ in $\K[X]$.
\end{proof}

Given a class of quadratic forms ${[A,2B,C] \in \cloksx}$ whose image through $\Psi$ is non-trivial, we now analyse the consequences on $A$. We start with the case of quadratic forms over $\K[X]$.

For the sake of reader-friendlyness, uppercase letters are used for quadratic forms $[A,2B,C]$ with coefficients in $\oks[X]$, whereas lowercase letters are used for quadratic forms $[a,2b,c]$ with coefficients in $\K[X]$.

\begin{prop} \label{rootnotsquare}
	If ${q \in \clkx \setminus \clkx^{\square}}$, and if $[a,2b,c]$ is a representative of $q$ with $a$ coprime to $f$, then for all ${\varepsilon \in \K^{\times}}$, there exists a root $\rho_{\varepsilon}$ of $f$ such that ${\varepsilon a(\rho_{\varepsilon}) \notin \K(\rho_{\varepsilon})^{\times \square}}$.
\end{prop}

\begin{proof}
	Let ${L := \faktor{\K[X]}{\left\langle f(X) \right\rangle}}$, and let
	$$\Psi \colon \faktor{\clkx}{\clkx^{\square}} \longrightarrow \faktor{L^{\times}}{\K^{\times}L^{\times \square}}$$
	be the Genus Theory homomorphism over $\K[X]$ defined in (\ref{mappsi}). Since $a$ is coprime to $f$, we have ${\Psi(q) = a^{-1}\K^{\times}L^{\times \square} = a\K^{\times}L^{\times \square}}$ by Proposition~\ref{coset}. On the other hand, $q$ is not a square, and $\Psi$ is injective by Theorem~\ref{inj}, hence ${a \notin \K^{\times}L^{\times \square}}$.
	
	Decompose ${f = \underset{i=1}{\overset{r}{\prod}} f_i}$ as a product of irreducible polynomials in $\oksx$. Recall that $f$ is assumed to be square-free, hence the $f_i$'s are all distinct. For all $i$, denote by $\rho_i$ a root of $f_i$. We then have
	$$L = \faktor{\K[X]}{\left\langle f(X) \right\rangle} \simeq \underset{i=1}{\overset{r}{\prod}} \faktor{\K[X]}{\left\langle f_i(X) \right\rangle} \simeq \underset{i=1}{\overset{r}{\prod}} \K(\rho_i).$$
	The image of ${a \in \K[X]}$ through these isomorphisms is ${(a(\rho_1),\ldots,a(\rho_r))}$.
	
	By contraposition, if there exists some ${\varepsilon \in \K^{\times}}$ such that $\varepsilon a(\rho_i)$ is a square in $\K(\rho_i)$ for all $i$, then $\varepsilon a$ is a square in $L$. Recall that $a$ is coprime to $f$, hence ${\varepsilon a(\rho_i) \neq 0}$ for all $i$, so ${\varepsilon a \in L^{\times \square}}$, and ${a \in \K^{\times}L^{\times \square}}$. Thus, $q$ must be a square in $\clkx$, and this concludes the proof.
\end{proof}

\begin{cor} \label{rootnotsquareint}
	Let ${q \in \cloksx}$. Let $[A,2B,C]$ be a representative of $q$ with $A$ coprime to $f$ in $\K[X]$ (see Lemma~\ref{intrep}). If $q$ is not a square in $\clkx$, then for all ${\varepsilon \in \oks^{\times}}$, there exists a root $\rho_{\varepsilon}$ of $f$ such that ${\varepsilon A(\rho_{\varepsilon})}$ is not a square in ${\rints{\K(\rho_{\varepsilon})}{\mathcal{S}_{\varepsilon}}}$ (where $\mathcal{S}_{\varepsilon}$ denotes the set of prime ideals above $\mathcal{S}$ in $\K(\rho_{\varepsilon})$).
\end{cor}

\begin{rem}
	With our choice of $\mathcal{S}$, requiring $q$ to be nonsquare in $\cloksx$ might not be enough to ensure that $q$ is not a square in $\clkx$, since the homomorphism from Corollary~\ref{coroksisokqf} is only injective. As seen in Remark~\ref{choice}, assuming that $\mathcal{S}$ also contains the prime ideals dividing $2\disc(f)$ solves this issue.
\end{rem}

\begin{proof}
	We apply Proposition~\ref{rootnotsquare}: for all ${\varepsilon \in \K^{\times}}$, there exists $\rho_{\varepsilon}$ a root of $f$ such that ${\varepsilon A(\rho_{\varepsilon}) \notin \K(\rho_{\varepsilon})^{\times\square}}$. On the other hand, for all ${\varepsilon \in \oks^{\times}}$, we have ${\varepsilon A(\rho_{\varepsilon}) \in \rints{\K(\rho_{\varepsilon})}{\mathcal{S}_{\varepsilon}}}$. Since this ring is integrally closed, $\varepsilon A(\rho_{\varepsilon})$ is a square in $\rints{\K(\rho_{\varepsilon})}{\mathcal{S}_{\varepsilon}}$ if and only if it is a square in $\K(\rho_{\varepsilon})$. We end up with the desired conclusion.
\end{proof}

The following Theorem relates Genus Theory over $\oks[X]$ to Genus Theory over $\oks$, providing a modular criterion for non-trivial specializations.

\begin{pgst} \label{theocritmod}
	Let ${f \in \ok[X]}$ be a square-free monic polynomial of odd degree at least $3$, and let $\mathcal{S}$ be a finite set of nonzero prime ideals of $\ok$ such that $\oks$ is a PID. Let ${q \in \cloksx}$, and assume that $q$ is not in the principal genus over $\K[X]$.
	
	Let $[A,2B,C]$ be a representative of $q$ with $A$ coprime to $f$ in $\K[X]$ (see Lemma~\ref{intrep}). Let ${n \in \oks}$ such that ${f(n) \neq 0}$. For all ${\varepsilon \in \faktor{\oks^{\times}}{\oks^{\times\square}}}$, let $\rho_{\varepsilon}$ be a root of $f$ such that $\varepsilon A(\rho_{\varepsilon})$ is not a square in $\rints{\K(\rho_{\varepsilon})}{\mathcal{S}_{\varepsilon}}$, as in Corollary~\ref{rootnotsquareint}.
	
	If, for all ${\varepsilon \in \faktor{\oks^{\times}}{\oks^{\times\square}}}$, there exists $\mathfrak{p}_{\varepsilon}$ a prime ideal in $\K(\rho_{\varepsilon})$ not above $\mathcal{S}$ such that $\mathfrak{p}_{\varepsilon}$ is inert in the extension ${\K(\rho_{\varepsilon}) \hookrightarrow \K\left(\rho_{\varepsilon}, \sqrt{\varepsilon A(\rho_{\varepsilon})}\right)}$ and ${n \equiv \rho_{\varepsilon} \pmod{\mathfrak{p}_{\varepsilon}}}$, then the quadratic form ${[A(n),2B(n),C(n)]}$ is not in the principal genus over $\oks$.
\end{pgst}

\begin{rem}
	There are only finitely many $\varepsilon$ to consider in Theorem~\ref{theocritmod}. Indeed, the quotient $\faktor{\oks^{\times}}{\oks^{\times\square}}$ has finite order by Dirichlet's Unit Theorem, extended to $\mathcal{S}$-integers by Chevalley and Hasse \cite[Theorem~3.12]{Nark}. More explicitly, one can write ${\oks^{\times} \simeq \mu \times \Z^r}$ where $\mu$ is the group of roots of unity, and then one has ${\faktor{\oks^{\times}}{\oks^{\times\square}} \simeq \faktor{\mu}{2\mu} \times \left(\faktor{\Z}{2\Z}\right)^r}$, which is finite.
\end{rem}

\begin{rem}
	If we further assume that the prime ideal $\mathfrak{p}_{\varepsilon}$ involved in Theorem~\ref{theocritmod} has inertia degree $1$ over $\oks$, then the congruence relation ${n \equiv \rho_{\varepsilon} \pmod{\mathfrak{p}_{\varepsilon}}}$ indeed has a solution ${n \in \oks}$.
\end{rem}

\begin{rem}
	By Chebotarev's Density Theorem, there are infinitely many prime ideals in $\rint{\K(\rho_{\varepsilon})}$ which are inert in the Galois extension $\K(\rho_{\varepsilon}) \hookrightarrow \K\left(\rho_{\varepsilon},\sqrt{\varepsilon A(\rho_{\varepsilon})}\right)$, hence removing a finite number of prime ideals ensures that there will still exist (infinitely many) inert prime ideals in $\rints{\K(\rho_{\varepsilon})}{\mathcal{S}_{\varepsilon}}$.
\end{rem}

\begin{proof}
	Let ${n \in \oks}$ be such that ${f(n) \neq 0}$, and let ${M_n := \faktor{\oks}{f(n)\oks}}$. To compute the image of ${q_n := [A(n),2B(n),C(n)] \in \cloks}$ by the Genus map $$\psi_n \colon \faktor{\cloks}{\cloks^{\square}} \longrightarrow \faktor{M_n^{\times}}{\oks^{\times}M_n^{\times \square}},$$ we would need $A(n)$ and $f(n)$ to be coprime, but this has no reason to happen. Nevertheless, we know from Lemma~\ref{coprime} and Remark~\ref{remsl2} that there exists another representative of $q_n$ whose first coefficient $A_n$ is nonzero and coprime to $f(n)$, and such that ${A_n = A(n)\alpha ^2 + 2B(n)\alpha\gamma + C(n)\gamma^2}$ for some coprime $\mathcal{S}$-integers $\alpha$ and $\gamma$. By Proposition~\ref{coset}, we have ${\psi_n(q_n) = A_n^{-1} \oks^{\times} M_n^{\times \square} = A_n \oks^{\times} M_n^{\times \square}}$, hence $q_n$ is in the principal genus if and only if ${\Big(A_n + f(n)\oks\Big) \in \oks^{\times} M_n^{\times \square}}$. This means that for all ${\varepsilon \in \faktor{\oks^{\times}}{\oks^{\times\square}}}$, we must show that ${\Big(\varepsilon A_n + f(n)\oks\Big) \notin M_n^{\times \square}}$.
	
	Fix some ${\varepsilon \in \faktor{\oks^{\times}}{\oks^{\times\square}}}$. Let $\mathfrak{p}_{\varepsilon}$ be a prime ideal in $\K(\rho_{\varepsilon})$ satisfying the hypotheses. Since ${n \equiv \rho_{\varepsilon} \pmod{\mathfrak{p}_{\varepsilon}}}$, we have ${f(n) \equiv 0 \pmod{\mathfrak{p}_{\varepsilon}}}$, that is, $\mathfrak{p}_{\varepsilon}$ divides $f(n)\rints{\K(\rho_{\varepsilon})}{\mathcal{S}_{\varepsilon}}$. This gives a ring homomorphism
	$$M_n = \faktor{\oks}{f(n)\oks} \longrightarrow \faktor{\rints{\K(\rho_{\varepsilon})}{\mathcal{S}_{\varepsilon}}}{f(n)\rints{\K(\rho_{\varepsilon})}{\mathcal{S}_{\varepsilon}}} \longrightarrow \faktor{\rints{\K(\rho_{\varepsilon})}{\mathcal{S}_{\varepsilon}}}{\mathfrak{p}_{\varepsilon}}.$$
	By contraposition, if $\varepsilon A_n$ is a square in $M_n$, then $\varepsilon A_n$ is a square in $\faktor{\rints{\K(\rho_{\varepsilon})}{\mathcal{S}_{\varepsilon}}}{\mathfrak{p}_{\varepsilon}}$. We show that the latter condition is not possible.
	
	Since $\mathfrak{p}_{\varepsilon}$ is inert in the quadratic extension ${\K(\rho_{\varepsilon}) \hookrightarrow \K\left(\rho_{\varepsilon},\sqrt{\varepsilon A(\rho_{\varepsilon})}\right)}$, the quantity $\varepsilon A(\rho_{\varepsilon})$ cannot be a square modulo $\mathfrak{p}_{\varepsilon}$. Moreover, as ${n \equiv \rho_{\varepsilon} \pmod{\mathfrak{p}_{\varepsilon}}}$, we have ${\varepsilon A(\rho_{\varepsilon}) \equiv \varepsilon A(n) \pmod{\mathfrak{p}_{\varepsilon}}}$, hence $A(n)$ is invertible modulo $\mathfrak{p}_{\varepsilon}$. Next, ${f(n) \equiv 0}$ implies ${B^2(n) \equiv A(n)C(n) \pmod{\mathfrak{p}_{\varepsilon}}}$, hence
	$$\varepsilon A_n = \varepsilon (A(n)\alpha ^2 + 2B(n)\alpha\gamma + C(n)\gamma^2) \equiv \varepsilon A(n)\left(\alpha + \frac{B(n)}{A(n)}\gamma\right)^2 \pmod{\mathfrak{p}_{\varepsilon}}.$$
	Since $A_n$ is coprime to $f(n)$, $A_n$ is invertible modulo $\mathfrak{p}_{\varepsilon}$, hence ${\alpha + \frac{B(n)}{A(n)}\gamma}$ is invertible too and ${\varepsilon A(n) \equiv \varepsilon A_n \left(\alpha + \frac{B(n)}{A(n)}\gamma\right)^{-2} \pmod{\mathfrak{p}_{\varepsilon}}}$. We deduce from that equation that $\varepsilon A_n$ is not a square modulo $\mathfrak{p}_{\varepsilon}$, hence modulo $f(n)$. We infer that ${{\Big(\varepsilon A_n + f(n)\oks\Big)} \notin M_n^{\times \square}}$ for all ${\varepsilon \in \faktor{\oks^{\times}}{\oks^{\times\square}}}$, hence ${\psi_n(q_n) \neq \oks^{\times}M_n^{\times\square}}$ and we are done.
\end{proof}

\begin{cor} \label{infty}
	With the same notations as above, let ${q \in \cloksx}$. Recall that for ${n \in \oks}$, $ev_n$ is the specialization map defined by (\ref{eval}).
	
	If $q$ is not in the principal genus over $\K[X]$, then there exist infinitely many ${n \in \oks}$ such that the specialized class of quadratic forms $ev_n(q)$ is non trivial in $\cloks$.
\end{cor}

\begin{proof}
	Let $\rho$ be a root of $f$ in some extension, and let $\nu$ be a non-square element of $\K(\rho)$. In view of the PGS Theorem~\ref{theocritmod} and the Chinese Remainder Theorem, and because the set $\mathcal{S}$ is finite, it is enough to show that there exist infinitely many different prime ideals $\mathfrak{p}$ of $\rint{\K(\rho)}$ satisfying:
	\begin{enumerate}
		\item[$\bullet$] $\mathfrak{p}$ is inert in $\rint{\K\left(\rho,\sqrt{\nu}\right)}$;
		\item[$\bullet$] $\mathfrak{p}$ has inertia degree $1$ over $\ok$ (this ensures that ${n \equiv \rho \pmod{\mathfrak{p}}}$ has a solution ${n \in \ok}$).
	\end{enumerate}
	The first point is guaranteed by Chebotarev's Density Theorem: the set of prime ideals of $\rint{\K(\rho)}$ inert in some (Galois) quadratic extension has Dirichlet density $\frac{1}{2}$. For the second point, it is a standard fact that the set of prime ideals of $\rint{\K(\rho)}$ having inertia degree $1$ over $\ok$ has Dirichlet density $1$ (see for example \cite[7.2.1, Corollary~3]{Nark}). Therefore, we may choose our prime ideals among a set of Dirichlet density $\frac{1}{2}$, hence there are infinitely many such ideals, as desired.
\end{proof}

\subsection{Density of non-trivial specializations in $\oks$} \label{secdens}

As in the previous Subsections, let ${f \in \ok[X]}$ be a square-free monic polynomial of odd degree at least $3$, and let $\mathcal{S}$ be a finite set of nonzero prime ideals of $\ok$ such that $\oks$ is a PID. Let ${q \in \cloksx}$, and assume that $q$ is not in the principal genus over $\K[X]$. Then the PGS Theorem~\ref{theocritmod} and its Corollary~\ref{infty} tell us that the set of $\mathcal{S}$-integers ${n \in \oks}$ such that the specialized class of quadratic forms ${ev_n(q) \in \cloks}$ is non-trivial is infinite. The present Section aims at estimating the density of this set in $\oks$. In general, there are various notions of density one can choose, depending on the problem one considers. Therefore, we give a list of properties that our density should satisfy.

\begin{defi} \label{defidensity}
	Denote by $\mathscr{P}(\oks)$ the set of subsets of $\oks$. In this paper, a \emph{density} is a map ${\delta \colon \mathscr{P}(\oks) \longrightarrow [0,1]}$ satisfying the following properties:
	\begin{enumerate}[label=(\density{d}{{\arabic*}})]
		\item $\delta(\oks) = 1$; \label{d1oks}
		\item if $P \subseteq Q$ in $\mathscr{P}(\oks)$, then $\delta(P) \leq \delta(Q)$; \label{d2incl}
		\item if $P, Q \subseteq \oks$, then $\delta(P \cup Q) \leq \delta(P) + \delta(Q)$; \label{d3union}
		\item $\delta$ is translation invariant, that is, $\delta(\alpha + P) = \delta(P)$ for all ${\alpha \in \oks}$ and ${P \in \mathscr{P}(\oks)}$;\label{d4transinv}
		\item for any $I$ ideal of $\oks$, we have ${\delta(I) = \frac{1}{N(I)}}$, where ${N(I) := \left\lvert \faktor{\oks}{I} \right\rvert = \left\lvert \faktor{\ok}{I \cap \ok} \right\rvert}$ is the norm of $I$. \label{d5ideal}
	\end{enumerate}
\end{defi}

\begin{ex}
	Let $\mathcal{S}_{\infty}$ be the set of archimedean places of $\ok$. For all ${\nu \in \mathcal{S} \cup \mathcal{S}_{\infty}}$, let $\K_{\nu}$ be the completion of $\K$ with respect to $\abs{\cdot}_{\nu}$. Then the map defined for all ${P \subseteq \oks}$ by
	$$\delta_{\mathcal{S}}(P) := \limsup_{r \in \K^{\times}} \frac{\#\set{n \in P \st \abs{n}_{\nu} \leq \abs{r}_{\nu} ~\forall \nu \in \mathcal{S} \cup \mathcal{S}_{\infty}}}{\#\set{n \in \oks \st \abs{n}_{\nu} \leq \abs{r}_{\nu} ~\forall \nu \in \mathcal{S} \cup \mathcal{S}_{\infty}}}$$
	is a density, according to \cite[Proposition~4.9 and Remark~4.10]{Longhiv6}.
\end{ex}

\paragraph{Notations.} Here, ${f \in \ok[X]}$ is a monic square-free polynomial of degree ${2g+1}$. Recall that $\faktor{\oks^{\times}}{\oks^{\times\square}}$ has finite order by Dirichlet's Unit Theorem. Let ${q \in \cloksx}$; as seen in Lemma~\ref{intrep}, there exists a representative $[A,2B,C]$ of $q$ such that ${A,B,C \in \oks[X]}$ and $A$ is coprime to $f$ in $\K[X]$. Recall that $[A,2B,C]$ is a primitive quadratic form of discriminant ${4B^2-4AC = 4f}$.

Moreover, we assume that $q$ is not a square in $\clkx$. By Corollary~\ref{rootnotsquareint}, for all ${\varepsilon \in \faktor{\oks^{\times}}{\oks^{\times\square}}}$, there exists a root $\rho_{\varepsilon}$ of $f$ such that $\varepsilon A(\rho_{\varepsilon})$ is not a square in $\rints{\K(\rho_{\varepsilon})}{\mathcal{S}_{\varepsilon}}$, where $\mathcal{S}_{\varepsilon}$ is the set of prime ideals of $\rint{\K(\rho_{\varepsilon})}$ lying over those of $\mathcal{S}$.

Our set of interest is
\begin{equation} \label{eqttriv}
	\mathcal{T}_{triv} := \set{n \in \oks \st [A(n),2B(n),C(n)] \text{~is~} \gltw \!{} \mhyphen \text{equivalent~over~} \oks \text{~to~} [1,0,-f(n)]}
\end{equation}
and our goal is to show that $\mathcal{T}_{triv}$ has density $0$.

\begin{defi} \label{defiepsnice}
	Let ${\varepsilon \in \faktor{\oks^{\times}}{\oks^{\times\square}}}$. We define the set ${\mathcal{P}_{\varepsilon} \subseteq \spec(\rints{\K(\rho_{\varepsilon})}{\mathcal{S}_{\varepsilon}})}$ by
	$$\mathcal{P}_{\varepsilon} := \set{\mathfrak{p} \in \spec(\rints{\K(\rho_{\varepsilon})}{\mathcal{S}_{\varepsilon}}) \st \begin{array}{c} \text{the~inertia~degree~of~} \mathfrak{p} \text{~over~} \K \text{~is~} 1 \\ \text{and~} \mathfrak{p} \text{~is~inert~in~the~extension~} \K(\rho_{\varepsilon}) \hookrightarrow \K\left(\rho_{\varepsilon},\sqrt{\varepsilon A(\rho_{\varepsilon})}\right) \end{array}}.$$
\end{defi}

The contrapositive of the PGS Theorem~\ref{theocritmod} directly implies the following inclusion.

\begin{prop} \label{propt}
	With the above notations, we have
	$${\mathcal{T}_{triv} ~ \subseteq ~ \set{n \in \oks \st f(n) = 0} ~~~ \bigcup ~~~ {\displaystyle \bigcup_{\varepsilon \in \faktor{\oks^{\times}}{\oks^{\times\square}}} ~ \bigcap_{\mathfrak{p} \in \mathcal{P}_{\varepsilon}} \set{n \in \oks \st n \not\equiv \rho_{\varepsilon} \pmod{\mathfrak{p}}}}}.$$
\end{prop}

\begin{defi} \label{defiuux}
	Let ${\varepsilon \in \faktor{\oks^{\times}}{\oks^{\times\square}}}$ and let $\rho_{\varepsilon}$ be the associated root of $f$. We set
	$$\mathcal{T}^{(\varepsilon)} := \bigcap_{\mathfrak{p} \in \mathcal{P}_{\varepsilon}} \set{n \in \oks \st n \not\equiv \rho_{\varepsilon} \pmod{\mathfrak{p}}} ~~~~~~\text{and}~~~~~~ \mathcal{T} := \bigcup_{\varepsilon \in \faktor{\oks^{\times}}{\oks^{\times\square}}} \mathcal{T}^{(\varepsilon)}.$$
	For all ${x \in \R_+^{\times}}$, we also set
	$$\mathcal{T}^{(\varepsilon)}_x := \bigcap_{\mathfrak{p} \in \mathcal{P}_{\varepsilon}, N(\mathfrak{p}) \leq x} \set{n \in \oks \st n \not\equiv \rho_{\varepsilon} \pmod{{\mathfrak{p}}}},$$
	where $N(\mathfrak{p})$ is the norm of the ideal $\mathfrak{p}$.
\end{defi}

Clearly, for all ${x \in \R_+^{\times}}$, we have ${\mathcal{T}^{(\varepsilon)} \subseteq \mathcal{T}^{(\varepsilon)}_x}$.

\begin{theo} \label{theodeltaleq}
	Let ${f \in \ok[X]}$ be a square-free monic polynomial of odd degree at least $3$, let $\mathcal{S}$ be a finite set of nonzero prime ideals of $\ok$ such that $\oks$ is a PID, let $\delta$ be a density, and let ${q = [A,2B,C] \in \cloksx}$ with $A$ coprime to $f$ in $\K[X]$ (see Lemma~\ref{intrep}). Assume that $q$ is not in the principal genus over $\K[X]$.
	
	Then the density of the set $\mathcal{T}_{triv}$ of ${n \in \oks}$ such that the specialized class of quadratic forms ${[A(n),2B(n),C(n)] \in \cloks}$ is trivial satisfies
	$$\delta(\mathcal{T}_{triv}) ~ \leq ~ \sum_{\varepsilon \in \faktor{\oks^{\times}}{\oks^{\times\square}}} ~ \prod_{\mathfrak{p} \in \mathcal{P}_{\varepsilon}} \left(1-\frac{1}{N(\mathfrak{p})}\right).$$
\end{theo}

\begin{proof}
	Recall that ${\mathcal{T}_{triv} \subseteq \mathcal{T} \cup \set{n \in \oks \st f(n) = 0}}$ by Proposition \ref{propt}, where $\mathcal{T}$ comes from Definition~\ref{defiuux}. Hence, ${\delta(\mathcal{T}_{triv}) \leq \delta(\mathcal{T})}$.
	
	Moreover,
	$$\mathcal{T} = \bigcup_{\varepsilon \in \faktor{\oks^{\times}}{\oks^{\times\square}}} \mathcal{T}^{(\varepsilon)} \subseteq \bigcup_{\varepsilon \in \faktor{\oks^{\times}}{\oks^{\times\square}}} \mathcal{T}^{(\varepsilon)}_x$$ for all ${x \in \R^{\times}_+}$, hence ${\displaystyle \delta(\mathcal{T}_{triv}) \leq \sum_{\varepsilon \in \faktor{\oks^{\times}}{\oks^{\times\square}}} \delta\left(\mathcal{T}^{(\varepsilon)}_x\right)}$ by properties \ref{d2incl} and \ref{d3union}.
	
	By Definition~\ref{defiepsnice}, if ${\mathfrak{p} \in \mathcal{P}_{\varepsilon}}$, then it has inertia degree $1$ above $\K$, hence the congruence ${n \equiv \rho_{\varepsilon} \pmod{\mathfrak{p}}}$ has a solution ${n \in \oks}$. Therefore, the congruence relation ${n \not\equiv \rho_{\varepsilon} \pmod{\mathfrak{p}}}$ has ${N(\mathfrak{p})-1}$ solutions modulo $\mathfrak{p}$. Using the Chinese Remainder Theorem, since all non-zero prime ideals of the Dedekind domain $\rints{\K(\rho_{\varepsilon})}{\mathcal{S}_{\varepsilon}}$ are maximal, we know that the system of congruences ${(n \not\equiv \rho_{\varepsilon} \pmod{\mathfrak{p}})_{\mathfrak{p} \in \mathcal{P}_{\varepsilon}, N(\mathfrak{p}) \leq x}}$ has ${\displaystyle \prod_{\mathfrak{p} \in \mathcal{P}_{\varepsilon}, N(\mathfrak{p}) \leq x} (N(\mathfrak{p})-1)}$ solutions modulo ${\displaystyle \prod_{\mathfrak{p} \in \mathcal{P}_{\varepsilon}, N(\mathfrak{p}) \leq x} \mathfrak{p}}$. Denote by $\mathcal{R}^{(\varepsilon)}_x$ a set of lifts of those solutions to $\oks$. Then, we can write
	
	\begin{align*}
	\mathcal{T}^{(\varepsilon)}_x & = \bigsqcup_{\alpha \in \mathcal{R}^{(\varepsilon)}_x} \set{n \in \oks \st n \equiv \alpha \pmod{\prod_{\mathfrak{p} \in \mathcal{P}_{\varepsilon}, N(\mathfrak{p}) \leq x} \mathfrak{p}}} \\
	& = \bigsqcup_{\alpha \in \mathcal{R}^{(\varepsilon)}_x} \left(\alpha + \prod_{\mathfrak{p} \in \mathcal{P}_{\varepsilon}, N(\mathfrak{p}) \leq x} \mathfrak{p}\right).
	\end{align*}
	
	Since ${\#\mathcal{R}^{(\varepsilon)}_x = {\displaystyle \prod_{\mathfrak{p} \in \mathcal{P}_{\varepsilon}, N(\mathfrak{p}) \leq x} (N(\mathfrak{p})-1)}}$, we have
	$$\delta\left(\mathcal{T}^{(\varepsilon)}_x\right) \leq \prod_{\mathfrak{p} \in \mathcal{P}_{\varepsilon}, N(\mathfrak{p}) \leq x} \left(1-\frac{1}{N(\mathfrak{p})}\right)$$
	by properties \ref{d4transinv} and \ref{d5ideal}. This being true for all ${x \in \R^{\times}_+}$, we infer that
	$$\delta(\mathcal{T}_{triv}) ~ \leq ~ \sum_{\varepsilon \in \faktor{\oks^{\times}}{\oks^{\times\square}}} ~ \lim_{x \to +\infty} \delta\left(\mathcal{T}^{(\varepsilon)}_x\right) ~ \leq ~ \sum_{\varepsilon \in \faktor{\oks^{\times}}{\oks^{\times\square}}} ~ \prod_{\mathfrak{p} \in \mathcal{P}_{\varepsilon}} \left(1-\frac{1}{N(\mathfrak{p})}\right).$$
\end{proof}

The only remaining step is to show that ${\displaystyle \prod_{\mathfrak{p} \in \mathcal{P}_{\varepsilon}} \left(1-\frac{1}{N(\mathfrak{p})}\right) = 0}$ for all ${\varepsilon \in \faktor{\oks^{\times}}{\oks^{\times\square}}}$. Actually, this will be a consequence of the fact that the set $\mathcal{P}_{\varepsilon}$ has strictly positive Dirichlet density.

\begin{theo} \label{theodens1}
	Let ${f \in \ok[X]}$ be a square-free monic polynomial of odd degree at least $3$, let $\mathcal{S}$ be a finite set of nonzero prime ideals of $\ok$ such that $\oks$ is a PID, let $\delta$ be a density, and let ${q = [A,2B,C] \in \cloksx}$.
	
	If $q$ is not in the principal genus over $\K[X]$, then the set
	$$\mathcal{T}_{triv} = \set{n \in \oks \st [A(n),2B(n),C(n)] \text{~is~} \gltw \!{} \mhyphen \text{equivalent~over~} \oks \text{~to~} [1,0,-f(n)]}$$
	has density ${\delta(\mathcal{T}_{triv}) = 0}$ in $\oks$.
\end{theo}

\begin{proof}
	 In view of Theorem~\ref{theodeltaleq}, it remains to compute the product ${\displaystyle \prod_{\mathfrak{p} \in \mathcal{P}_{\varepsilon}} \left(1-\frac{1}{N(\mathfrak{p})}\right)}$ for all ${\varepsilon \in \faktor{\oks^{\times}}{\oks^{\times\square}}}$. Let us fix such an $\varepsilon$. As already noticed in the proof of Corollary~\ref{infty}, the set $\mathcal{P}_{\varepsilon}$ from Definition~\ref{defiepsnice} has Dirichlet density $\frac{1}{2}$, being the intersection of a set of Dirichlet density $1$ (the prime ideals having inertia degree $1$ over $\K$) and another one of density $\frac{1}{2}$ (the prime ideals inert in some quadratic extension). This implies that ${{\displaystyle \sum_{\mathfrak{p} \in \mathcal{P}_{\varepsilon}} \frac{1}{N(\mathfrak{p})^s}} \underset{s \to 1^+}{\longrightarrow} +\infty}$, since otherwise, the Dirichlet density of $\mathcal{P}_{\varepsilon}$ would vanish by definition.
	
	Beside this, one can order prime ideals of $\rints{\K(\rho_{\varepsilon})}{\mathcal{S}_{\varepsilon}}$ such that the sequence $(N(\mathfrak{p}_j))_{j}$ is increasing. Applying \cite[Lemma~IV.4.4]{Janusz} with ${u_j := N(\mathfrak{p}_j)}$ leads to
	$$\sum_{\mathfrak{p} \in \mathcal{P}_{\varepsilon}} \frac{1}{N(\mathfrak{p})^s} = -\sum_{\mathfrak{p} \in \mathcal{P}_{\varepsilon}} \ln\left(1 - \frac{1}{N(\mathfrak{p})^s}\right) + \underset{s \to 1^+}{O}(1)$$
	for all ${s > 1}$. Taking the exponential, we get
	\begin{align*}
	\prod_{\mathfrak{p} \in \mathcal{P}_{\varepsilon}} \left(1-\frac{1}{N(\mathfrak{p})^s}\right) & = \exp\left(\sum_{\mathfrak{p} \in \mathcal{P}_{\varepsilon}} \ln\left(1 - \frac{1}{N(\mathfrak{p})^s}\right)\right) \\
	& = \exp\left(-\sum_{\mathfrak{p} \in \mathcal{P}_{\varepsilon}} \frac{1}{N(\mathfrak{p})^s}\right)\exp\left(\underset{s \to 1^+}{O}(1)\right) \underset{s \to 1^+}{\longrightarrow} 0.
	\end{align*}
	The conclusion is now straightforward, applying Theorem~\ref{theodeltaleq}.
\end{proof}

\end{sloppypar}

 \providecommand{\andname}{\&} \providecommand{\bibsep}{0cm}
  \providecommand{\toappear}{to appear} \providecommand{\noopsort}[1]{}

\textsc{William Dallaporta}, Institut de Mathématiques de Toulouse, Université de Toulouse, CNRS UMR 5219, 118 route de Narbonne, 31062 Toulouse Cedex 9, France.

E-mail address: \url{william.dallaporta@laposte.net}


\begin{thebibliography}{{Tow}80}

\bibitem[AP00]{AgboolaPappas}
A.~Agboola and G.~Pappas.
\newblock Line bundles, rational points and ideal classes.
\newblock {\em Math. Res. Lett.}, 7(5-6):709--717, 2000.

\bibitem[Bue89]{Buell}
Duncan~A. Buell.
\newblock {\em Binary quadratic forms. {Classical} theory and modern
  computations}.
\newblock New York, NY etc.: Springer-Verlag, 1989.

\bibitem[Can87]{Cantor}
David~G. Cantor.
\newblock Computing the {Jacobian} of a hyperelliptic curve.
\newblock {\em Math. Comput.}, 48:95--101, 1987.

\bibitem[Coh93]{Cohen}
Henri Cohen.
\newblock {\em A course in computational algebraic number theory}, volume 138
  of {\em Graduate Texts in Mathematics}.
\newblock Springer-Verlag, Berlin, 1993.

\bibitem[{Cox}13]{Cox2}
David~A. {Cox}.
\newblock {\em {Primes of the form \(x^2+ny^2\). Fermat, class field theory,
  and complex multiplication.}}
\newblock Hoboken, NJ: John Wiley \& Sons, 2nd edition, 2013.

\bibitem[Dal21]{moi}
William Dallaporta.
\newblock Recovering the {Picard} group of quadratic algebras from {Wood}'s
  binary quadratic forms.
\newblock arXiv, 2021.
\newblock \url{https://arxiv.org/abs/2111.13422} (to appear in \textit{International Journal of Number Theory}).

\bibitem[DL23]{Longhiv6}
Luca Demangos and Ignazio Longhi.
\newblock Densities on {Dedekind} domains, completions and {Haar} measure.
\newblock arXiv, 2023.
\newblock \url{https://arxiv.org/abs/2009.04229v6} (to appear in \textit{Mathematische Zeitschrift}).

\bibitem[FPS97]{Poonen}
E.~V. Flynn, Bjorn Poonen, and Edward~F. Schaefer.
\newblock Cycles of quadratic polynomials and rational points on a genus-2
  curve.
\newblock {\em Duke Math. J.}, 90(3):435--463, 1997.

\bibitem[Gil21]{Jean}
Jean Gillibert.
\newblock From {Picard} groups of hyperelliptic curves to class groups of
  quadratic fields.
\newblock {\em Trans. Am. Math. Soc.}, 374(6):3919--3946, 2021.

\bibitem[GL12]{GillibertLevin}
Jean Gillibert and Aaron Levin.
\newblock Pulling back torsion line bundles to ideal classes.
\newblock {\em Math. Res. Lett.}, 19(6):1171--1184, 2012.

\bibitem[Jan96]{Janusz}
Gerald~J. Janusz.
\newblock {\em Algebraic number fields.}, volume~7 of {\em Grad. Stud. Math.}
\newblock Providence, RI: AMS, American Mathematical Society, 2nd ed. edition,
  1996.

\bibitem[{Kne}82]{Kneser}
Martin {Kneser}.
\newblock {Composition of binary quadratic forms}.
\newblock {\em {J. Number Theory}}, 15:406--413, 1982.

\bibitem[Mum84]{Mumford}
David Mumford.
\newblock {\em Tata lectures on theta. {II}: {Jacobian} theta functions and
  differential equations. {With} the collaboration of {C}. {Musili}, {M}.
  {Nori}, {E}. {Previato}, {M}. {Stillman}, and {H}. {Umemura}}, volume~43 of
  {\em Prog. Math.}
\newblock Birkh{\"a}user, Cham, 1984.

\bibitem[Nar04]{Nark}
W{\l}adys{\l}aw Narkiewicz.
\newblock {\em Elementary and analytic theory of algebraic numbers}.
\newblock Springer Monogr. Math. Berlin: Springer, 3rd ed. edition, 2004.

\bibitem[Sch95]{Sch}
Edward~F. Schaefer.
\newblock 2-descent on the {Jacobians} of hyperelliptic curves.
\newblock {\em J. Number Theory}, 51(2):219--232, 1995.

\bibitem[{Ses}58]{Seshadri}
C.~S. {Seshadri}.
\newblock {Triviality of vector bundles over the affine space \(K^2\)}.
\newblock {\em {Proc. Natl. Acad. Sci. USA}}, 44:456--458, 1958.

\bibitem[Sol94]{Soleng}
Ragnar Soleng.
\newblock Homomorphisms from the group of rational points on elliptic curves to
  class groups of quadratic number fields.
\newblock {\em J. Number Theory}, 46(2):214--229, 1994.

\bibitem[SS11]{SivSol}
Tormod~Kalberg Sivertsen and Ragnar Soleng.
\newblock Hyperelliptic curves and homomorphisms to ideal class groups of
  quadratic number fields.
\newblock {\em J. Number Theory}, 131(12):2303--2309, 2011.

\bibitem[{Tow}80]{Towber}
Jacob {Towber}.
\newblock {Composition of oriented binary quadratic form-classes over
  commutative rings}.
\newblock {\em {Adv. Math.}}, 36:1--107, 1980.

\bibitem[Woo11]{Wood}
Melanie~Matchett Wood.
\newblock Gauss composition over an arbitrary base.
\newblock {\em Adv. Math.}, 226(2):1756--1771, 2011.

\end{thebibliography}
\end{document}